\documentclass[12pt,twoside]{amsart}
\usepackage{amssymb,amsmath,amsfonts,amsthm}
\usepackage{mathrsfs}
\usepackage[X2,T1]{fontenc}
\usepackage[applemac]{inputenc}
\usepackage{cite}
\usepackage{latexsym}
\usepackage{verbatim}
\usepackage{soul}
\usepackage{ulem}
\usepackage[dvipsnames]{xcolor}
\def\Im{\mathop{\rm Im}\nolimits}

\def\supp{\mathop{\rm supp}\nolimits}

\def\Im{\mathop{\rm Im}\nolimits}

\def\supp{\mathop{\rm supp}\nolimits}

\def\R{\mathbb R}

\def\N{\mathbb N}

\def\ds{\displaystyle}

\newcommand\dslash{d\llap {\raisebox{.9ex}{$\scriptstyle-\!$}}}
\usepackage{colortbl}

\newcommand{\beqsn}{\arraycolsep1.5pt\begin{eqnarray*}}
\newcommand{\eeqsn}{\end{eqnarray*}\arraycolsep5pt}
\newcommand{\beqs}{\arraycolsep1.5pt\begin{eqnarray}}
\newcommand{\eeqs}{\end{eqnarray}\arraycolsep5pt}

\newtheorem{theorem}{Theorem}
\newtheorem{lemma}{Lemma}

\newtheorem{proposition}{Proposition}
\newtheorem{definition}{Definition}

\newtheorem{remark}{Remark}

\catcode`\@=11

\renewcommand{\section}%
   {\setcounter{equation}{0}\@startsection {section}{1}{\z@}{-3.5ex plus -1ex
  minus -.2ex}{2.3ex plus .2ex}{\Large\bf}}

\title[ ]{On the Cauchy problem for $p$-evolution equations with variable coefficients: a necessary condition for Gevrey well-posedness}

\author[A. Arias Junior]{Alexandre Arias Junior}
\address{Dipartimento di Matematica ``G. Peano'' \\Universit\`a di Torino\\
	Via Carlo Alberto 10\\
	10123 Torino\\
	Italy}
\email{alexandre.ariasjunior@unito.it}

\author[A. Ascanelli]{Alessia Ascanelli}
\address{Dipartimento di Matematica ed Informatica\\Universit\`a di Ferrara\\
Via Machiavelli 30\\
44121 Ferrara\\
Italy}
\email{alessia.ascanelli@unife.it}

\author[M. Cappiello]{Marco Cappiello}
\address{Dipartimento di Matematica ``G. Peano'' \\Universit\`a di Torino\\
Via Carlo Alberto 10\\
10123 Torino\\
Italy}
\email{marco.cappiello@unito.it}

\begin{document}


\begin{abstract}
In this paper we consider a class of $p$-evolution equations of arbitrary order with variable coefficients depending on time and space variables $(t,x)$. We prove necessary conditions on the decay rates of the coefficients for the well-posedness of the related Cauchy problem in Gevrey spaces. 
\end{abstract}

\maketitle

\section{Introduction}\label{section_introduction}
The main concern in this paper is the Gevrey well-posedness of the Cauchy problem
\begin{equation}\label{cpp}
	\begin{cases}
		Pu(t,x) = 0, \quad (t,x) \in [0,T]\times \R, \\
		u(0,x) = \phi(x), \quad x \in \R
	\end{cases}
\end{equation} 
where, for a fixed $p \in \N, p \geq 2$, $P$ is a non-kowalewskian linear evolution operator of the form
\begin{equation}\label{full operator}
P = D_t + a_p(t)D^{p}_{x} + \sum_{j = 1}^{p} a_{p-j}(t,x)D_x^{p-j}, \quad (t,x)\in[0,T]\times\R, \quad D=-i\partial.
\end{equation}
We assume that $a_p \in C([0,T]; \R)$, $a_p(t)\neq 0$ $\forall t\in [0,T]$ and $a_{p-j} \in C([0,T];\mathcal{B}^{\infty}(\R)), j=1,\ldots,p$, where $\mathcal{B}^{\infty}(\R)$ stands for the space of complex-valued functions bounded on $\R$ together with all their derivatives. The operator $P$ is known in the literature as a {\it $p-$evolution operator with real characteristics} (cf. \cite{mizohata2014}). To point out some outstanding particular cases, we recover a Schr\"odinger-type operator when $p = 2$ and linearized KdV-type equations when $p = 3$, cf. \cite{AAC3evolquasilin}. Also for higher values of $p$  the linearizations of several dispersive evolution equations can be written using an operator of the form \eqref{full operator}, see e.g. \cite{PZ} and the references therein.
The Cauchy problem \eqref{cpp} for \eqref{full operator} has been intensively investigated and well understood in $H^\infty(\R)= \bigcap_{m \in \R} H^m(\R)$, see \cite{ascanelli_chiara_zanghirati_2012, ascanelli_chiara_zanghirati_necessary_condition_for_H_infty_well_posedness_of_p_evo_equations, CicRei, ichinose_remarks_cauchy_problem_schrodinger_necessary_condition, I2, KB} and in the Schwartz spaces $\mathscr{S}(\R), \mathscr{S}'(\R)$, see \cite{ACJee}. 
 
In fact, our assumption $a_p(t) \in \R$ agrees with the necessary condition for well-posedness in $H^\infty(\R)$ given by Theorem $3$ on page $31$ of \cite{mizohata2014}: $\Im a_p(t)\leq 0\ \forall t\in [0,T]$.  Under this condition, if the coefficients $a_j$ of the lower order terms are all real-valued, then problem \eqref{cpp} is $L^2$ well-posed; well-posedness of \eqref{cpp} may fail if $\Im a_j\neq 0$ for some $1\leq j\leq p-1.$

A necessary condition on the coefficient $a_{p-1}(t,x)$ of the subprincipal part of $P$ for well-posedness in $H^\infty(\R)$ has been given in \cite{ichinose_remarks_cauchy_problem_schrodinger_necessary_condition} when $p = 2$ and in \cite{ascanelli_chiara_zanghirati_necessary_condition_for_H_infty_well_posedness_of_p_evo_equations} for general $p \geq 3$; it reads as follows:\\
{\it{if \eqref{cpp} is well-posed in $H^\infty(\R)$, then
there exist $M,N>0$ such that 
\begin{equation}\label{cn}\min_{0\leq\tau\leq t\leq T}\int_{-\varrho}^\varrho 
\Im a_{p-1}(t,x+p a_p(\tau)\theta)d\theta\leq M\log(1+\varrho)+N,\qquad 
\forall \varrho>0, \ \forall x\in\R.\end{equation}}}
We recall that dealing with one space dimension, for $p=2$ condition \eqref{cn} is also sufficient for $H^\infty$ well-posedness and with $M=0$ it is necessary and sufficient for $L^2$ well-posedness, cf.\cite{I2}. 

The condition \eqref{cn} clearly implies that if $a_{p-1}$ does not depend on $x$, the problem \eqref{cpp} is not $H^\infty$ well-posed. On the other hand, $x$-dependent coefficients with bounded primitive are allowed. For instance, the operator $D_t + D^2_x + i\cos x \ D_{x}$ fulfills \eqref{cn} with $M=0$ and the associated Cauchy problem is $L^2$ well-posed, cf. \cite{mizohata2014}. 
Nevertheless, if we test condition \eqref{cn}  on the model operator 
\begin{equation}\label{eq_model_operator}
	 D_t + D^p_x + i\langle x \rangle^{-\sigma}D^{p-1}_{x}, \quad \langle x\rangle:=\sqrt{1+x^2},
\end{equation}
we immediately realize that for $\sigma\in (0,1)$ $H^\infty$ well-posedness does not hold. In other words, if $a_{p-1}$ is vanishing for $|x| \to \infty$, then its decay must be fast enough to hope for $H^\infty$ well posedness.

In the study of sufficient conditions for the $H^{\infty}$ well-posedness of \eqref{cpp} it is customary to impose some pointwise decay conditions as $|x| \to \infty$ on the coefficients $a_j(t,x)$, $1\leq j\leq p-1$, see \cite{CicRei, KB} for $p = 2$ and \cite{ascanelli_chiara_zanghirati_2012} for general $p \geq 3$. These decay conditions are of the form $$|\Im a_{j}(t,x)|\leq C_T\langle x\rangle^{-\frac{j}{p-1}}, 1\leq j\leq p-1,$$ matched with suitable decay conditions on the $x$-derivatives of $a_j$.

 \medskip
Coming now to well-posedness in Gevrey type spaces, the known results are restricted to the case $p=2$ for necessary conditions, see \cite{dreher, CRnec}, and to the cases $p = 2$ and $p = 3$ for sufficient conditions, see \cite{KB} and \cite{AACgev} respectively. Concerning sufficient conditions, as in the $H^\infty$ case, they are expressed as decay conditions for $|x| \to \infty$ on the coefficients of the lower order terms. As far as we know, there are no results concerning necessary conditions for Gevrey well-posedness for $p \geq 3$.  
The above mentioned results are settled in the following class of Gevrey functions:
$$
\mathcal{H}^{\infty}_{\theta}(\R) := \bigcup_{\rho > 0} H^{0}_{\rho;\theta}(\R), \quad H^{0}_{\rho;\theta}(\R) := \{ u \in \mathscr{S}'(\R) \colon e^{\rho\langle \xi \rangle^{\frac{1}{\theta}}} \widehat{u}(\xi) \in L^{2}(\R) \}
$$
endowed with the norm
$$
\|u\|_{H^{0}_{\rho;\theta}(\R)} := \| e^{\rho\langle \cdot \rangle^{\frac{1}{\theta}}} \widehat{u}(\cdot)\|_{L^{2}(\R)},
$$
where $\rho > 0$ and $\widehat{u}$ denotes the Fourier transform of $u$. The space $\mathcal{H}^{\infty}_{\theta}(\R)$ is related to Gevrey class of functions in the following sense:
$$
G^{\theta}_{0}(\R) \subset \mathcal{H}^{\infty}_{\theta}(\R) \subset G^{\theta}(\R),
$$
where $G^{\theta}(\R)$ denotes the space of all smooth functions $f$ such that for some $C>0$
$$
\sup_{\alpha \in \N_0} \sup_{x \in \R} |\partial^{\alpha}_{x}f(x)| C^{-\alpha} \alpha!^{-\theta} < +\infty,
$$
and $G^{\theta}_{0}(\R)$ is the space of all compactly supported functions contained in $G^{\theta}(\R)$.

\medskip 

The definition of well-posedness in $\mathcal{H}^{\infty}_{\theta}(\R)$ for the Cauchy problem \eqref{cpp} reads as follows:
\begin{definition}\label{def_WP}
	We say that the Cauchy problem \eqref{cpp} is well-posed in $\mathcal{H}^{\infty}_{\theta}(\R)$ if for any given $\rho_0 > 0$ there exist $\rho > 0$ and $C:=C(\rho,T) > 0$ such that for all $\phi \in H^{0}_{\rho_0;\theta}(\R)$ there exists a unique solution $u \in C^{1}([0,T];H^{0}_{\rho;\theta}(\R))$ and the following energy inequality holds 
	$$
	\| u(t,\cdot) \|_{H^{0}_{\rho;\theta}(\R)} \leq C \|\phi\|_{H^{0}_{\rho_0;\theta}(\R)}, \quad \forall t \in [0,T].
	$$
\end{definition}
\noindent To  explain the available results in the literature in a simple way it is convenient to use the 
model operator \eqref{eq_model_operator} for $p = 2$ and $p = 3$ with $\sigma \in (0,1)$. For results concerning the more general operator \eqref{full operator} we refer to \cite[Theorem 1.1]{KB} in the case $p=2$ and to \cite[Theorem 1]{AACgev} in the case $p=3$. 

\begin{itemize}
\item By \cite{dreher} and \cite{KB}, for the operator $D_t+D_x^2+i\langle x\rangle^{-\sigma} D_x$ we have:
$$
\sigma\in (0,1)\, \text{and} \, 1-\sigma \leq \frac{1}{\theta}\iff \text{well-posedness in}\, \mathcal{H}^\infty_\theta(\R),\  \theta > 1 .
$$
In the limit case $1-\sigma =\frac{1}{\theta}$ we have local in time well-posedness, whereas when $1-\sigma<\frac{1}{\theta}$ we get it on the whole interval $[0,T]$. The number $1/(1-\sigma)$ works so as a threshold for the Gevrey indices $\theta$ for which  well-posedness results can be found. The case $1-\sigma>\frac{1}{\theta}$ has been also investigated in \cite{ACR}, where under additional exponential decay conditions on the datum $g$ a solution in suitable Gevrey classes has been obtained.
\item By \cite{AACgev}, for the operator $D_t+D_x^3+i\langle x\rangle^{-\sigma} D^2_x$ we have:
$$
\sigma \in \left(\frac{1}{2}, 1\right)\, \text{and} \, 2(1-\sigma)\leq \frac{1}{\theta}  \implies \, \text{well-posedness in}\, \mathcal{H}^\infty_\theta(\R),\ \theta > 1.
$$
Again, in the limit case $2(1-\sigma)= \frac{1}{\theta} $ we have  local in time well-posedness, whereas when $2(1-\sigma) < \frac{1}{\theta} $ we get it on the whole interval $[0,T]$. 
\end{itemize}
Thus, in the case $p=3$ the following two questions arise: 
\begin{itemize}
\item [\bf Q1)] What happens when the decay rate $\sigma$ is less than or equal to $\frac{1}{2}$?
\item[\bf Q2)] What happens when $2(1-\sigma)>\frac{1}{\theta} $?
\end{itemize} 

The main goal of this manuscript is to answer the two above questions, and at the same time  to generalize both the questions and the answers to the more general operator \eqref{full operator}. 
Namely, we give a necessary condition on the Gevrey index $\theta$ and on the  decay rates of the imaginary parts of the coefficients of \eqref{full operator} for the well-posedness of the Cauchy problem \eqref{cpp} in $\mathcal{H}^\infty_\theta(\R), \theta >1$. Our main theorem reads as follows.

\begin{theorem}\label{main_theorem}
	Let $\theta > 1$. Let $P$ an operator of the form \eqref{full operator} with $a_p \in C([0,T]; \R)$ and $a_p(t) \neq 0 \, \forall t \in [0,T]$, and assume that the coefficients $a_{p-j}$ satisfy the following conditions:
	\begin{itemize}
		\item [(i)] there exist $R, A > 0$ and $\sigma_{p-j} \in [0,1]$, $j = 1, \ldots, p-1,$ such that 
		$$
		Im\, a_{p-j}(t,x) \geq A \langle x \rangle^{-\sigma_{p-j}}, \quad x > R \, (\textrm{or} \, \, \, x <-R), \, t \in [0,T], \, j=1,\ldots,p-1; $$
		\item [(ii)] there exists $C > 0$ such that
		$$
		|\partial^{\beta}_{x}a_{p-j}(t,x)| \leq C^{\beta + 1} \beta! \langle x \rangle^{-\beta}, \quad x \in \R, \, t \in [0,T],\quad j = 1, \ldots, p.
		$$
	\end{itemize}
 If the Cauchy problem \eqref{cpp} is well-posed in $\mathcal{H}^{\infty}_{\theta}(\R)$, then 
	\begin{equation}\label{equation_thesis_main_theorem}
		\Xi := \max_{j=1,\ldots,p-1} \{(p-1)(1-\sigma_{p-j}) - j + 1\} \leq \frac{1}{\theta}.
	\end{equation}
\end{theorem}

\begin{remark}
Notice that in the case $p=2$, the condition i) in Theorem  \ref{main_theorem} is equivalent in one space dimension to the \textit{slow decay condition} assumed in \cite{dreher}. 
\end{remark}

\begin{remark}
	Let us notice that since $(p-1)(1-\sigma_{p-1}) > 0$ we always have $\Xi > 0$. We also point out the following inequalities
	$$
	(p-1)(1-\sigma_{p-j}) -j +1 \geq 1 \iff \sigma_{p-j} \leq \frac{p-1-j}{p-1},
	$$
	$$
	(p-1)(1-\sigma_{p-j}) -j +1 \leq 0 \iff \sigma_{p-j} \geq \frac{p-j}{p-1}.
	$$
	Therefore, as a consequence of \eqref{equation_thesis_main_theorem}, when $\theta > 1$ we conclude the following:
	\begin{itemize}
		\item If $\sigma_{p-j} \leq \frac{p-1-j}{p-1}$ for some $j=1,\ldots, p-1,$ the Cauchy problem is not well-posed in $\mathcal{H}^{\infty}_{\theta}(\R)$;
		
		\item If $\sigma_{p-j} \geq \frac{p-j}{p-1}$ for some $j=1,\ldots, p-1,$ then the power $\sigma_{p-j}$ has no effect on the $\mathcal{H}^{\infty}_{\theta}$ well-posedness;
		
		\item If $\sigma_{p-j} \in \left( \frac{p-1-j}{p-1},  \frac{p-j}{p-1}\right)$ for some $j=1,\ldots, p-1,$ then the power $\sigma_{p-j}$ imposes the restriction 
		$$
		(p-1)(1-\sigma_{p-j}) - j + 1 \leq \frac{1}{\theta}
		$$
		for the indices $\theta$ where $\mathcal{H}^{\infty}_{\theta}$ well-posedness can be found.  
	\end{itemize}

\noindent In particular, we conclude that a first order coefficient of the type $a_{1}(t,x) = i \langle x \rangle^{-\sigma_1}, \sigma_1 \in (0,1)$ cannot affect the $\mathcal{H}^{\infty}_{\theta}$ well-posedness of \eqref{cpp} for all $\theta > 1$, whereas $a_{p-j}(t,x) = i\langle x \rangle^{-\sigma_{p-j}}$, $j = 1, \ldots, p-2$, can compromise well-posedness provided that $\sigma_{p-j}$ is close enough to zero.
	\end{remark} 
	
	\begin{remark} We also notice that if at least one of the coefficients $a_{p-j}(t,x)$ decays less rapidly than any negative power of $\langle x \rangle$, e.g. if $a_{p-j}(t,x)=  i(\log(\sqrt{2+x^2}))^{-1}$ for some $j$, then $a_{p-j}$ fulfills the assumption (ii) for any $\sigma_{p-j} > 0$. So, for any fixed $\theta > 1$ we can find $\sigma_{p-j} \in (0,1)$ so that $(p-1)(1-\sigma_{p-j}) -j + 1 > \theta^{-1}$. Hence, there is no well-posedness in $\mathcal{H}^{\infty}_{\theta}(\R), \theta > 1$ in this case. This opens the new question of finding a suitable functional setting where the Cauchy problem is well-posed in this situation.
\end{remark}

Considering the model operator $D_t + D^3_x + i\langle x \rangle^{-\sigma}D^{2}_{x}$, $\sigma \in (0,1)$, from Theorem \ref{main_theorem} and the above remark we obtain the answer to the two questions posed before.

\begin{itemize}
\item [\bf A1)] There is no well-posedness in $\mathcal{H}^\infty_\theta(\R)$ of the Cauchy problem for $D_{t} + D^{3}_{x} + i\langle x \rangle^{-\sigma}D^{2}_{x}$ when $\sigma \leq \frac{1}{2}$.
\item[\bf A2)] If $\sigma \in (1/2,1)$, the Cauchy problem associated with $D_{t} + D^{3}_{x} + i\langle x \rangle^{-\sigma}D^{2}_{x}$ is not well-posed in $\mathcal{H}^{\infty}_{\theta}(\R)$ when $2(1-\sigma)>\frac{1}{\theta} $ . Thus, $2(1-\sigma)$ is the threshold for the Gevrey indices $\theta$ for which well-posedness results in $\mathcal{H}^\infty_\theta(\R)$ can be obtained.
\end{itemize}

\begin{remark}
By the known results for $p=2, p=3$ and by Theorem \ref{main_theorem}, we conjecture that we can obtain well-posedness for \eqref{cpp} in $\mathcal{H}^\infty_\theta(\R)$ for all $\theta \in (1, \Xi^{-1})$ for a generic $p$. Notice that the latter interval becomes more and more narrow as $p$ increases. We intend to treat this problem in a future paper.
\end{remark}

The paper is organized as follows. In Section \ref{section_preliminaries} we recall some classical definitions and preliminary results that will be needed for the subsequent parts of the work. Section \ref{section_idea_of_the_proof} is devoted to define the principal tools for the proof of Theorem \ref{main_theorem} and to explain the strategy we will follow to prove it. Section \ref{section_estimates_from_below} contains several technical estimates needed to obtain our main result. Finally, in Section \ref{section_proof_of_main_thm}, we present the proof of Theorem \ref{main_theorem}. 

\section{Preliminaries}\label{section_preliminaries}

In this section we fix some notation and recall the basic definitions and results we will employ in the subsequent sections. Throughout the paper we shall denote respectively by $\langle \cdot, \cdot \rangle$ and $\| \cdot \|$ the scalar product and the norm in $L^2(\R)$. For functions depending on $(t,x)$ appearing in the next sections, $\langle \cdot , \cdot \rangle$ and $\| \cdot \|$ will always denote the scalar product and the norm in $L^2(\R_x)$ for a fixed $t \in [0,T]$. 

Given $m\in\R$, we denote by $S^{m}_{0,0}(\R)$ the space of all functions $p \in C^{\infty}(\R^2)$ such that for any $\alpha, \beta \in \N_0$ the following estimate holds 
$$
|\partial^{\alpha}_{\xi} \partial^{\beta}_{x} p(x,\xi)| \leq C_{\alpha,\beta} \langle \xi \rangle^{m}
$$
for a positive constant $C_{\alpha, \beta}.$
The topology of the space $S^{m}_{0,0}(\R)$ is induced by the following family of seminorms
$$
|p|^{(m)}_\ell := \max_{\alpha \leq \ell, \beta \leq \ell} \sup_{x, \xi \in \R} |\partial^{\alpha}_{\xi} \partial^{\beta}_{x} p(x,\xi)| \langle \xi \rangle^{-m}, \quad p \in S^{m}_{0,0}(\R), \, \ell \in \N_0.
$$

As usual we associate to every symbol $p \in S^{m}_{0,0}(\R)$ the continuous operator on the Schwartz space of rapidly decreasing functions $p(x,D): \mathscr{S}(\R) \to \mathscr{S}(\R)$, known as pseudodifferential operator, given by 
$$
p(x,D) u(x) = \int e^{i\xi x} p(x,\xi) \widehat{u}(\xi) \dslash\xi, \quad u \in \mathscr{S}(\R),
$$
where $\dslash \xi := \frac{d\xi}{2\pi}$. Sometimes we will write $p(x,D)={\rm{op}}(p(x,\xi))$. The next Theorem \ref{theorem_Calderon_Vaillancourt} gives the action of operators coming from symbols $S^{m}_{0,0}(\R)$ in the standard Sobolev spaces $H^{s}(\R)$, $s \in \R$, defined by
$$
H^{s}(\R) := \{ u \in \mathscr{S}'(\R) \colon \langle \xi \rangle^{s} \widehat{u}(\xi) \in L^2(\R)\}, \quad \|u\|_{H^{s}(\R)} := \|\langle \xi \rangle^{s} u\|.
$$
We recall that, when $s$ is a positive integer we can replace $\|u\|_{H^{s}(\R)}$ by the equivalent norm 
$$
\|u\|_{s} := \sum_{j=0}^{s} \|D^{j}_{x}u\|.
$$

\begin{theorem}\label{theorem_Calderon_Vaillancourt}[Calder\'on-Vaillancourt]
	Let $p \in S^{m}_{0,0}(\R)$. Then for any real number $s \in \R$ there exist $\ell := \ell(s,m) \in \N_0$ and $C:= C_{s,m} > 0$ such that 
	\begin{equation*}
		\| p(\cdot,D)u \|_{H^{s}(\R)} \leq C |p|^{(m)}_{\ell} \| u \|_{H^{s+m}(\R)}, \quad \forall \, u \in H^{s+m}(\R).
	\end{equation*}
	Besides, when $m = s = 0$ we can replace $|p|^{(m)}_{\ell}$ by
	\begin{equation*}
		\max_{\alpha, \beta \leq 2} \sup_{x, \xi \in \R} |\partial^{\alpha}_{\xi} \partial^{\beta}_{x} p(x,\xi)|.
	\end{equation*}
\end{theorem}
For a proof of Theorem \ref{theorem_Calderon_Vaillancourt} we address the reader to Theorem $1.6$ on page $224$ of \cite{Kumano-Go}. 
Now we consider the algebra properties of $S^{m}_{0,0}(\R)$ with respect to the composition of operators. Let $p_j \in S^{m_j}_{0,0}(\R)$, $j = 1, 2$, and define 
\begin{align}\label{eq_symbol_of_composition}
q(x,\xi) &= Os- \iint e^{-iy \eta} p_1(x,\xi+\eta)p_2(x+y,\xi) dy \dslash\eta \\
		 &= \lim_{\varepsilon \to 0} \iint e^{-iy\eta} p_1(x,\xi+\eta)p_2(x+y,\xi) e^{-\varepsilon^2y^2} e^{-\varepsilon^2\eta^2} dy \dslash\eta. \nonumber
\end{align} 
Then we have the following theorem (for a proof see Lemma $2.4$ on page $69$ and Theorem $1.4$ on page $223$ of \cite{Kumano-Go}).

\begin{theorem}
	Let $p_j \in S^{m_j}_{0,0}(\R)$, $j = 1, 2$, and consider $q$ defined by \eqref{eq_symbol_of_composition}. Then $q \in S^{m_1+m_2}_{0,0}(\R)$ and $q(x,D) = p_1(x,D) p_2(x,D)$. Moreover, the symbol $q$ has the following asymptotic expansion
	\begin{align*}
		q(x,\xi) = \sum_{\alpha < N} \frac{1}{\alpha!} \partial^{\alpha}_{\xi}p_{1}(x,\xi)D^{\alpha}_{x}p_2(x,\xi) + r_N(x,\xi), 
	\end{align*} 
	where 
	$$
	r_N(x,\xi) = N\int_{0}^{1} \frac{(1-\theta)^{N-1}}{N!} \, Os - \iint e^{-iy\eta} \partial^{N}_{\xi} p_1(x,\xi+\theta\eta) D^{N}_{x} p_2(x+y,\xi)  dy\dslash\eta \, d\theta,
	$$
	and the seminorms of $r_N$ may be estimated in the following way: for any $\ell_{0} \in \N_0$ there exists $\ell_{1} := \ell_1(\ell_{0}) \in \N_0$ such that 
	$$
	|r_N|^{(m_1+m_2)}_{\ell_0} \leq C_{\ell_{0}} |\partial^{N}_{\xi}p_1|^{(m_1)}_{\ell_{1}} |\partial^{N}_{x}p_2|^{(m_2)}_{\ell_{1}}.
	$$
\end{theorem} 

The last theorem that we recall is a version of the so-called sharp G{\aa}rding inequality which takes into account the polynomial behavior at infinity on $|x|$ of the symbol. To state this result it is convenient to introduce the $SG$ symbol classes. Given $m_1, m_2 \in \R$, we say that a smooth function $p \in C^{\infty}(\R^{2})$ belongs to the class $SG^{m_1,m_2}(\R^{2})$ if for every $\alpha, \beta \in \N_0$ there exists a positive constant $C_{\alpha,\beta} > 0$ such that 
$$
|\partial^{\alpha}_{\xi} \partial^{\beta}_{x} p(x,\xi)| \leq C_{\alpha,\beta} \langle \xi \rangle^{m_1-\alpha} \langle x \rangle^{m_2-\beta}.
$$
Let $p \in SG^{m_1,m_2}(\R^{2})$, then we consider the following special type of Friedrichs' symmetrization 
\begin{equation}\label{eq_friedrics_part_symbol}
p_{F,L}(x,\xi) = Os-\iint e^{-iy\eta} p_{F}(\xi+\eta, x+y, \xi) dy \dslash\eta,
\end{equation}
where 
$$
p_{F}(\xi,x',\xi') = \int F(x',\xi,\zeta) p(x',\zeta) F(x',\xi',\zeta) d\zeta, \quad x', \xi, \xi' \in \R,
$$
$$
F(x', \xi, \zeta) = q(\langle x' \rangle^{\frac{1}{2}}\langle \xi \rangle^{-\frac{1}{2}} (\xi - \zeta)) \langle \xi \rangle^{-\frac{1}{4}} \langle x' \rangle^{\frac{1}{4}}, \quad x',\xi, \zeta \in \R,
$$
for some even cutoff function $q \in C^{\infty}_{0}(\R)$ such that $\int q^2 = 1$ (cf. Definition $11$ of \cite{AC_SG_Sharp_Garding}). Then we have the following result (see Theorem $4$ and Proposition $6$ of \cite{AC_SG_Sharp_Garding}).

\begin{theorem}\label{theorem_SG_sharp_garding}
	Let $p \in SG^{m_1,m_2}(\R)$ and let moreover $p_{F,L}$ be the symbol defined by $\eqref{eq_friedrics_part_symbol}$. Then $p_{F,L} \in SG^{m_1,m_2}(\R)$ and $p-p_{F,L} \in SG^{m_1-1,m_2-1}(\R)$. Moreover, if $p(x,\xi) \geq 0$ then $\langle p_{F,L}(x,D) u, u \rangle \geq 0$ for all $u \in \mathscr{S}(\R)$.
\end{theorem}

\section{Idea of the proof}\label{section_idea_of_the_proof}

To prove our result we follow an argument inspired by \cite{dreher}. We shall prove that if the Cauchy problem \eqref{cpp} is well-posed in $\mathcal{H}^{\infty}_{\theta}(\R)$, then the denial of \eqref{equation_thesis_main_theorem} leads to a contradiction. With this idea in mind let us start by defining the main ingredients to get the desired contradiction. Consider a Gevrey cutoff function $h \in G^{\theta_h}_{0}(\R)$ for some $\theta_h > 1$ close to $1$ such that
$$
h(x) =
\begin{cases}
	1, \quad |x| \leq \frac{1}{2}, \\
	0, \quad |x| \geq 1.
\end{cases}
$$
For a sequence $\{\nu_k\}$ of positive real numbers such that $\nu_k \to \infty$ as $k \to \infty$, we define the following sequence of symbols 
\begin{equation}\label{eq_def_localization}
	w_{k}(x,\xi) = h\left( \frac{ x-4\nu^{p-1}_{k} }{ \nu^{p-1}_{k} } \right) h\left(  \frac{\xi - \nu_k}{ \frac{1}{4}\nu_{k} } \right).
\end{equation}
Note that $w_{k}(x,\xi)$ is a symbol localized around the bicharacteristic curve of $\xi^{p}$ passing through the point $(0,\nu_k)$ at some fixed time $t$ (in this case $t = 4/p$). Indeed, the bicharacteristic curve of $\xi^p$ (usually called {\it Hamilton flow} generated by the operator $D_t+D_x^p$) passing through a point $(x_0, \xi_0) \in \R^{2}$,  is the solution of 
\begin{equation}\label{flow}
\begin{cases}
	x'(t) = p\xi(t)^{p-1},\,\, x(0) = x_0, \\
	\xi'(t) = 0,\,\, \xi(0) = \xi_0,
\end{cases}
\end{equation}
that is
$(x(t),\xi(t)) = (x_0+pt\xi_0^{p-1}, \xi_0).$

\begin{remark}\label{remark_x_xi_equiv_sigma_k}
	On the support of $w_{k}(\cdot, \cdot)$ we have that $\xi$ is comparable with $\nu_k$ and $x$ is comparable with $\nu^{p-1}_{k}$. Indeed, if $\xi \in \supp w_{k}(x,\cdot)$ then
	$$
	 	|\xi - \nu_k| \leq \frac{1}{4}\nu_k  \iff \frac{3\nu_k}{4} \leq \xi \leq \frac{5\nu_k}{4},
	$$
	for all $k\in \N_0$. Similarly, if $x \in \supp w_k(\cdot, \xi)$ then 
	$$
		|x-4\nu^{p-1}_k| \leq \nu^{p-1}_k \iff  3\nu^{p-1}_k \leq x \leq 5\nu_k^{p-1},
	$$
	for all $k \in \N_0$. 
\end{remark}

Consider now $\phi \in G^{\theta}(\R)$ such that 
$$
\widehat{\phi}(\xi) = e^{-2\rho_0 \langle \xi \rangle^{\frac{1}{\theta}}},
$$
for some $\rho_0 > 0$ and define $\phi_k(x) = \phi(x-4\nu^{p-1}_{k})$ for all $k \in \N_0$. From the assumed $\mathcal{H}^{\infty}_{\theta}$ well-posedness, let $u_k \in C^1([0,T];H^{0}_{\rho;\theta}(\R))$ be the solution of \eqref{cpp} with initial datum $\phi_k\in H^{0}_{\rho_0;\theta}(\R)$. Then for $\lambda \in (0,1)$ and $\theta_1 > \theta_h$ (still close to $1$) to be chosen later, we define
\begin{equation}\label{eq_definition_of_N_k}
N_k = \lfloor \nu^{\frac{\lambda}{\theta_1}}_{k} \rfloor
\end{equation}
and then we introduce the energy 
\begin{equation}\label{eq_def_energy}
E_{k}(t) = \sum_{\alpha \leq N_k, \beta \leq N_k} \frac{1}{(\alpha!\beta!)^{\theta_1}} \| w^{(\alpha\beta)}_{k} (x,D) u_k(t,x) \| = \sum_{\alpha \leq N_k, \beta \leq N_k} E_{k,\alpha,\beta}(t),
\end{equation}
where 
\begin{equation*}
	w^{(\alpha\beta)}_{k}(x,\xi) = h^{(\alpha)}\left( \frac{x-4\nu^{p-1}_{k}}{ \nu^{p-1}_{k} } \right)
	h^{(\beta)} \left( \frac{\xi - \nu_k}{ \frac{1}{4}\nu_{k} } \right).
\end{equation*}

The next lemma, whose proof follows from Remark \ref{remark_x_xi_equiv_sigma_k} and a simple computation, gives estimates for the norms of $w_k$.

\begin{lemma}\label{lemma_estimates_w_k}
	Let $\alpha, \beta, \gamma, \delta, \mu, \ell \in \N_{0}$. Then $w^{(\alpha\beta)}_{k} \in S^{0}_{0,0}(\R^2)$ and there exists $C := C(\theta_h)$ such that 
	$$
	|\xi^{\mu} \partial^{\delta}_{x} \partial^{\gamma}_{\xi} w^{(\alpha\beta)}_{k}(x,\xi)|^{(0)}_{\ell} \leq C^{\alpha+\beta+\gamma+\delta+\mu+\ell+1} (\alpha!\beta!\gamma!\delta!\ell!^2)^{\theta_h} \nu^{\mu-\gamma}_{k} \nu^{-\delta(p-1)}_{k}.
	$$
\end{lemma}

Assuming that the Cauchy problem \eqref{cpp} is $\mathcal{H}^{\infty}_{\theta}$ well-posed, by the energy estimate for the solution (see Definition 1) we simply obtain a uniform upper bound for the energies $E_k(t)$ with respect to both $k \in \N_0$ and $t \in [0,T]$. Indeed, Calder\'on-Vaillancourt theorem implies
 (for the same constant $C$ of Lemma \ref{lemma_estimates_w_k}) that
 $$
 E_{k,\alpha,\beta}(t)\leq C^{\alpha+\beta+1}(\alpha!\beta!)^{\theta_h-\theta_1}  \|u_k(t)\| \leq C^{\alpha+\beta+1}(\alpha!\beta!)^{\theta_h-\theta_1}  \|u_k(t)\|_{H^0_{\rho,\theta}(\R)}
 $$
 and then from the $\mathcal{H}^{\infty}_{\theta}$ well-posedness 
 \begin{eqnarray}\label{Ekab}
 E_{k,\alpha,\beta}(t)&\leq& C_{T,\rho_0}C^{\alpha+\beta}(\alpha!\beta!)^{\theta_h-\theta_1}  \|\phi_k\|_{H^0_{\rho_0,\theta}(\R)}
 \\\nonumber
 &=&C_{T,\rho_0}C^{\alpha+\beta}(\alpha!\beta!)^{\theta_h-\theta_1}  \|\phi\|_{H^0_{\rho_0,\theta}(\R)}
 \\\nonumber&=& C_{T,\rho_0,\phi}C^{\alpha+\beta}(\alpha!\beta!)^{\theta_h-\theta_1} . 
 \end{eqnarray}
Recalling that $\theta_1 > \theta_h$, we conclude that
\begin{align}\label{eq_estimate_from_above_of_energy_k}
	E_{k}(t) &\leq C_{T,\rho_0,\phi}\sum_{\alpha \leq N_k, \beta \leq N_k}  C^{\alpha+\beta}(\alpha!\beta!)^{\theta_h-\theta_1}  \leq C_{T,\rho_0,\phi}\sum_{\alpha, \beta \geq 0}  C^{\alpha+\beta}(\alpha!\beta!)^{\theta_h-\theta_1}  \\ \nonumber
	&= C_{1} C_{T,\rho_0,\phi}, \quad \forall \, t \in [0,T], k \in \N_0.
\end{align}

The idea to get a contradiction is to use the energy method to obtain an estimate from below for $E_k(t)$, $t \in [0,T]$, of the type
$$
E_k(t) \geq f(\nu_k),
$$
where $f(\nu_k) \to +\infty$ if \eqref{equation_thesis_main_theorem} does not hold. Obtaining such estimate is the most involving part of the proof, so we shall dedicate the next section to deal with this problem. 

\section{Estimates from below for $E_k(t)$}\label{section_estimates_from_below}

We denote
\begin{equation}\label{vk}
v^{(\alpha\beta)}_{k}(t,x) = w^{(\alpha\beta)}_{k}(x,D)u_k(t,x).
\end{equation}
Then we have
$$
Pv^{(\alpha\beta)}_{k} = w^{(\alpha\beta)}_{k} \underbrace{P u_k}_{= 0} + [P, w^{(\alpha\beta)}_{k} ]u_k = [P, w^{(\alpha\beta)}_{k} ]u_k =: f^{(\alpha\beta)}_{k},
$$
and therefore 
\begin{align}\label{e}
	\|v^{(\alpha\beta)}_{k}\| &\partial_t \|v^{(\alpha\beta)}_{k}\| = \frac{1}{2}\partial_t \{ \| v^{(\alpha\beta)}_{k} \|^{2}\} \\\nonumber
	&= \text{Re}\, \langle \partial_t v^{(\alpha\beta)}_{k}, v^{(\alpha\beta)}_{k} \rangle \\\nonumber
	&= \text{Re}\, \langle if^{(\alpha\beta)}_{k}, v^{(\alpha\beta)}_{k} \rangle  - \underbrace{\text{Re}\, \langle ia_p(t)D_x^p v^{(\alpha\beta)}_{k}, v^{(\alpha\beta)}_{k} \rangle}_{=0} - \sum_{j=1}^{p} \text{Re}\, \langle ia_{p-j}(t,x) D^{p-j}_{x} v^{(\alpha\beta)}_{k}, v^{\alpha\beta}_{k} \rangle \\\nonumber
	&\geq - \| f^{(\alpha\beta)}_{k}\|\| v^{(\alpha\beta)}_{k} \| - \sum_{j=1}^{p-1} \text{Re}\, \langle ia_{p-j}(t,x) D^{p-j}_{x} v^{(\alpha\beta)}_{k}, v^{\alpha\beta}_{k} \rangle - C_{a_0} \|v^{(\alpha\beta)}_{k}\|^{2}.
\end{align} 

The terms $\text{Re}\, \langle ia_{p-j}(t,x) D^{p-j}_{x} v^{(\alpha\beta)}_{k}, v^{(\alpha\beta)}_{k} \rangle$ ($j=1, \ldots, p-1$) will provide the contradiction, whereas $\| f^{(\alpha\beta)}_{k}\|$ will be negligible in some sense. Being more precise, we shall obtain estimates from above for $\| f^{(\alpha\beta)}_{k}\|$ and absorb them into an estimate from below for the terms $\text{Re}\, \langle  ia_{p-j}(t,x) D^{p-j}_{x} v^{(\alpha\beta)}_{k}, v^{(\alpha\beta)}_{k} \rangle$. In the next two subsections we shall discuss these estimates.

\subsection{Estimate from below for $\text{Re} \, \langle ia_{p-j}(t,x) D^{p-j}_{x} v^{(\alpha\beta)}_{k} , v^{(\alpha\beta)}_{k} \rangle$.}\label{subs_stime}
For $j = 1, \ldots, p-1$ we write 
\begin{align*}
	- \text{Re}\, \langle &ia_{p-j} D^{p-j}_{x} v^{(\alpha\beta)}_{k} , v^{(\alpha\beta)}_{k} \rangle = \text{Re}\, \langle \text{Im}\, a_{p-j} D^{p-j}_{x} v^{(\alpha\beta)}_{k} , v^{(\alpha\beta)}_{k} \rangle - \text{Re} \, \langle i\text{Re}\,a_{p-j} D^{p-j}_{x} v^{(\alpha\beta)}_{k} , v^{(\alpha\beta)}_{k} \rangle \\
	&= \text{Re} \, \langle \text{Im}\, a_{p-j} D^{p-j}_{x} v^{(\alpha\beta)}_{k} , v^{(\alpha\beta)}_{k} \rangle + \frac{1}{2}\sum_{s=0}^{p-j-1} \binom{p-j}{s} \langle i D^{p-j-s}_{x}\text{Re}\, a_{p-j} D^{s}_{x} v^{(\alpha\beta)}_{k}, v^{(\alpha\beta)}_{k} \rangle.
\end{align*}
Then we consider the following cutoff functions 
\begin{equation}\label{eq_cutoff_functions_to_split_the_supports}
	\chi_{k}(\xi) = h \left( \frac{\xi - \nu_k}{ \frac{3}{4}\nu_{k}} \right), \quad \psi_{k}(x) = h \left( \frac{x-4\nu^{p-1}_{k}}{3\nu^{p-1}_{k}} \right).
\end{equation}
Note that on the support of $\psi_{k}(x) \chi_{k}(\xi)$ we have the following
$$
|\xi - \nu_k| \leq \frac{3}{4} \nu_{k} \iff \frac{\nu_k}{4} \leq \xi \leq \frac{7\nu_k}{4},
$$ 
$$
|x-4\nu^{p-1}_{k}| \leq 3\nu^{p-1}_{k}  \iff \nu^{p-1}_{k} \leq x \leq 7 \nu^{p-1}_{k},
$$
for all $k \in \N_0$. Therefore,
\begin{equation}\label{st}
(x,\xi) \in \supp \psi_{k}(x) \chi_{k}(\xi) \;\Rightarrow \; \xi^{p-j} \geq \frac{\nu^{p-j}_{k}}{4^{p-j}}, \quad \langle x \rangle^{-\sigma_{p-j}} \geq 7^{-\sigma_{p-j}} \langle \nu^{p-1}_{k} \rangle^{-\sigma_{p-j}},
\end{equation}
Denoting \begin{equation}\label{c_0}
	c_{p-j} = A\frac{7^{-\sigma_{p-j}}}{4^{p-j}},
\end{equation} we decompose the symbol of $\text{Im}\, a_{p-j}(t,x) D^{p-j}_{x}$ as follows 
\begin{align*}
	\text{Im}\,a_{p-j}(t,x) \xi^{p-j} &= c_{p-j}\langle \nu^{p-1}_{k} \rangle^{-\sigma_{p-j}} \nu^{p-j}_{k} + \left( \text{Im}\,a_{p-j}(t,x) \xi^{p-j} - c_{p-j}\langle \nu^{p-1}_{k} \rangle^{-\sigma_{p-j}} \nu^{p-j}_{k} \right)\\
	&= c_{p-j}\langle \nu^{p-1}_{k} \rangle^{-\sigma_{p-j}} \nu^{p-j}_{k} +  \left(\text{Im}\,a_{p-j}(t,x) \xi^{p-j} - c_{p-j}\langle \nu^{p-1}_{k} \rangle^{-\sigma_{p-j}} \nu^{p-j}_{k}\right)\psi_{k}(x) \chi_{k}(\xi) \\
	&+ \left( \text{Im}\,a_{p-j}(t,x) \xi^{p-j} - c_{p-j}\langle \nu^{p-j}_{k} \rangle^{-\sigma_{p-j}} \nu^{p-j}_{k}\right)\left(1-\psi_{k}(x)\chi_{k}(\xi)\right).
\end{align*}

Using assumption (i) we get for $k$ sufficiently large
\begin{equation}\label{uso i}
\text{Im}\, a_{p-j}(t,x) \xi^{p-j} \geq 
c_{p-j} \langle \nu^{p-1}_{k} \rangle^{-\sigma_{p-j}} \nu^{p-j}_{k} + I_{p-j,k}(x,\xi) + J_{p-j,k}(t,x,\xi), 
\end{equation}
where
\begin{eqnarray}
\label{i2k}
 I_{p-j,k}(x,\xi) &=& \left(A \langle x \rangle^{-\sigma_{p-j}} \xi^{p-j} - c_{p-j}\langle \nu^{p-1}_{k} \rangle^{-\sigma_{p-j}} \nu^{p-j}_{k}\right)\psi_{k}(x) \chi_{k}(\xi),
 \\\label{i3k}
 J_{p-j,k}(t,x,\xi) &=& \left( \text{Im}\, a_{p-j}(t,x) \xi^{p-j} - c_{p-j}\langle \nu^{p-1}_{k} \rangle^{-\sigma_{p-j}} \nu^{p-j}_{k}\right)\left(1-\psi_{k}(x)\chi_{k}(\xi)\right),
\end{eqnarray}
where $A$ is the constant appearing in condition i) of Theorem \ref{main_theorem}. \\
Let us immediately notice that by \eqref{st},\eqref{i2k} and the choice \eqref{c_0} of $c_{p-j}$ we have 
\begin{equation}\label{positive}
 I_{p-j,k}(x,\xi) \geq 0, \quad \forall\, (x,\xi)\in\R^2.
\end{equation}

Hence 
\begin{align}\label{3termini}
Re\, \langle &\text{Im}\, a_{p-j}(t,x) D^{p-j}_{x} v^{(\alpha\beta)}_{k}, v^{(\alpha\beta)}_{k} \rangle \geq Re\, \langle c_{p-j} \langle \nu^{p-1}_{k} \rangle^{-\sigma_{p-j}} \nu^{p-j}_{k} v^{(\alpha\beta)}_{k}, v^{(\alpha\beta)}_{k} \rangle
\\\nonumber
&+Re\, \langle I_{p-j,k}(x,D) v^{(\alpha\beta)}_{k}, v^{(\alpha\beta)}_{k} \rangle+Re\, \langle J_{p-j,k}(t,x,D) v^{(\alpha\beta)}_{k}, v^{(\alpha\beta)}_{k} \rangle
\\\nonumber
&= c_{p-j} \langle \nu^{p-1}_{k} \rangle^{-\sigma_{p-j}} \nu^{p-j}_{k} \|v^{(\alpha\beta)}_{k}\|^2
\\\nonumber
&+Re\, \langle I_{p-j,k}(x,D) v^{(\alpha\beta)}_{k}, v^{(\alpha\beta)}_{k} \rangle+Re\, \langle J_{p-j,k}(t,x,D) v^{(\alpha\beta)}_{k}, v^{(\alpha\beta)}_{k} \rangle
\\\nonumber
&\geq c_{p-j} 2^{-\sigma_{p-j}/2} \nu^{(p-1)(1-\sigma_{p-j})-j+1}_{k} \|v^{(\alpha\beta)}_{k}\|^2
\\\nonumber
&+ Re\, \langle I_{p-j,k}(x,D) v^{(\alpha\beta)}_{k}, v^{(\alpha\beta)}_{k} \rangle+Re\, \langle J_{p-j,k}(t,x,D) v^{(\alpha\beta)}_{k}, v^{(\alpha\beta)}_{k} \rangle .
\end{align}

From now on, we shall denote by $C$ a positive constant independent of $k,\alpha,\beta,N_k$ but possibly depending on $\phi,\rho_0,T$ and on the coefficients $a_{p-j}$. 

\begin{lemma}\label{lemma_useful_estimate}
	If \eqref{cpp} is well-posed in $\mathcal{H}^{\infty}_{\theta}(\R)$, then for any $M \in \N$ the following estimate holds:
	\begin{equation*}
		\|D^{r}_{x} v^{(\alpha\beta)}_{k}\| \leq C \nu^{r}_{k}\|v^{(\alpha\beta)}_{k}\|+ C^{\alpha+\beta+M+1}\{\alpha!\beta!\}^{\theta_h} M!^{2\theta_h - 1} \nu^{r-M}_{k}
	\end{equation*}
for some $C>0$ independent of $k$.
\end{lemma}
\begin{proof}
	We split the symbol of $D^{r}_{x}$ as
	$$
	\xi^{r} = \xi^{r}\chi_{k}(\xi) + \xi^{r}\{1-\chi_k(\xi)\}.
	$$ 
	Using the properties of the support of $\chi_k$ ($\xi$ comparable with $\nu_k$) and Calder\'on-Vaillancourt theorem we obtain
	\begin{equation}\label{eq_estimate_11}
		\|D^{r}_{x}\chi_{k}(D) v^{(\alpha\beta)}_{k}\| \leq C \nu^{r}_{k} \|v^{(\alpha\beta)}_{k}\|.
	\end{equation}
	
	Since the supports of $w^{(\alpha\beta)}_{k}(x,\xi)$ and of $1-\chi_{k}(\xi)$ are disjoint, recalling \eqref{vk} we get 
	\begin{align}\label{split}
		{\rm{op}}\left(\xi^{r}\{1-\chi_{k}(\xi)\}\right) w^{(\alpha\beta)}_{k}(x,D) = r^{(\alpha\beta)}_{k,M}(x,D), 
	\end{align} 
	where, defining $q_{k}(\xi) = \xi^{r}\{1-\chi_{k}(\xi)\}$,
	\begin{align*}
		r^{(\alpha\beta)}_{k,M}(x,\xi) = M\int_{0}^{1} \frac{(1-\theta)^{M-1}}{M!} \, Os - \iint e^{-iy\eta} \partial^{M}_{\xi} q_{k}(\xi+\theta\eta) D^{M}_{x}w^{(\alpha\beta)}_{k}(x+y,\xi)  dy\dslash\eta \, d\theta.
	\end{align*}
	We write $\partial^{M}_{\xi}q_k$ in the following way:
	\begin{equation}\label{eq_annoying_term}
	\partial^{M}_{\xi}q_k(\xi) =\left\{
	\begin{array}{ll}
	\{1-\chi_{k}(\xi)\}\ds\frac{r!}{(r-M)!} \xi^{r-M} - \sum_{\overset{M_1 + M_2 = M}{M_1\geq 1}}\frac{M!}{M_1!M_2!} \partial^{M_1}_{\xi}\chi_k(\xi) \frac{r!}{(r-M_2)!}\xi^{r-M_2}, & M \leq r,
	\\
	- \ds\sum_{\overset{M_1 + M_2 = M}{M_2\leq r}}\ds\frac{M!}{M_1!M_2!} \partial^{M_1}_{\xi}\chi_k(\xi) \frac{r!}{(r-M_2)!}\xi^{r-M_2}, & M > r.
	\end{array}\right.
	\end{equation}
 If $M > r$ we see that $\supp{\partial^{M}_{\xi} q_k} \subset \supp{\chi_k}$, so we can estimate the seminorms of $r^{\alpha\beta}_{k,M}$ as follows: for every $\ell_0 \in \N_0$ there exists $\ell_1 := \ell_1(\ell_0)$ such that  
	\begin{align*}
		|r^{(\alpha\beta)}_{k,M}|^{(0)}_{\ell_0} \leq C(\ell_0) \frac{M}{M!} |\partial^{M}_{\xi} q_{k}|^{(0)}_{\ell_1} |\partial^{M}_{x}w^{(\alpha\beta)}_{k}|^{(0)}_{\ell_1}.
	\end{align*}
	From Lemma \ref{lemma_estimates_w_k} we get
	$$
	|\partial^{M}_{x}w^{(\alpha\beta)}_{k}|^{(0)}_{\ell_1} \leq C^{\ell_1+\alpha+\beta+M+1} \ell_1!^{2\theta_h} \{\alpha!\beta!M!\}^{\theta_h} \nu^{-M(p-1)}_{k}.
	$$
	On the other hand (using the support of $\chi_k$) we have 
	$$
	|\partial^{M}_{\xi}q_{k}(\xi)|^{(0)}_{\ell_1} \leq C^{\ell_1+M+1}\ell_{1}!^{\theta_h}M!^{\theta_h} \nu^{r-M}_{k}.
	$$
	Therefore 
	$$
	|r^{(\alpha\beta)}_{k,M}|^{(0)}_{\ell_0} \leq C^{\alpha+\beta+M+1} \{\alpha!\beta!\}^{\theta_h} M!^{2\theta_h-1} \nu^{r-M}_{k} \nu_{k}^{-M(p-1)}.
	$$
	
	When $M \leq r$, from \eqref{eq_annoying_term} we have $\partial^{M}_{\xi}q_{k}(\xi) = \{1-\chi_{k}(\xi)\}r!(r-M)!^{-1}\xi^{r-M} + b_k(\xi)$ where $\supp{b_k} \subset \supp{\chi_k}$. Since the part of $r^{(\alpha\beta)}_{k,M}$ involving $b_k$ can be estimated as before, we only explain how to treat the part regarding $\{1-\chi_{k}(\xi)\}r!(r-M)!^{-1}\xi^{r-M}$ which will be denoted by $a^{(\alpha\beta)}_{k,M}$ in the following. Writing 
	$$
	(\xi+\theta\eta)^{r-M} = \sum_{r_1 + r_2 = r-M} \frac{(r-M)!}{r_1!r_2!} \xi^{r_1} (\theta\eta)^{r_2}, 
	$$
	integration by parts gives
	\begin{align*}
		a^{(\alpha\beta)}_{k,M}(x,\xi) = &\frac{r!}{(r-M)!}\sum_{r_1 + r_2=r-M} \frac{(r-M)!}{r_1!r_2!} M\int_{0}^{1} \frac{(1-\theta)^{M-1}}{M!} \theta^{r_2} \\
		&Os - \iint e^{-iy\eta} \{1-\chi_k\}(\xi+\theta\eta)  D^{M+r_2}_{x}w^{(\alpha\beta)}_{k}(x+y,\xi)\xi^{r_1}  dy\dslash\eta \, d\theta.
	\end{align*}
	In this way, for any $\ell_0 \in \N_0$ there exists $\ell_1 := \ell_1(\ell_0)$ such that 
	\begin{align*}
		|a^{(\alpha\beta)}_{k,M}(x,\xi)|^{(0)}_{\ell_0} &\leq  \frac{r!M}{M!} \sum_{r_1+r_2=r-M} \frac{1}{r_1!r_2!} |\{1-\chi_k\}(\xi)|^{(0)}_{\ell_1} |\xi^{r_1}D^{M+r_2}_{x} w^{(\alpha\beta)}_{k}(x,\xi)|^{(0)}_{\ell_1} \\
		&\leq \frac{r!M}{M!} \sum_{r_1+r_2=r-M} \frac{1}{r_1!r_2!}  C^{\alpha+\beta+r+1}_{\ell_0} \{\alpha!\beta!M!r_2!\}^{\theta_h} \nu^{-(p-1)(M+r_2)}_{k} \nu^{r_1}_{k} \\
		&\leq C^{\alpha+\beta+M+1}_{\ell_1, r} \{\alpha!\beta!\}^{\theta_h} M!^{2\theta_h-1} \nu^{r-M}_{k} \nu^{-(p-1)M}_{k}.
	\end{align*} 
	Thus, we also have 
	\begin{equation}\label{remai}
	|r^{(\alpha\beta)}_{k,M}|^{(0)}_{\ell_0} \leq C^{\alpha+\beta+M+1} \{\alpha!\beta!\}^{\theta_h} M!^{2\theta_h-1} \nu^{r-pM}_{k},
	\end{equation}
	when $M \leq r$. Therefore, from \eqref{vk},\eqref{split}, Calder\'on-Vaillancourt Theorem, \eqref{remai} and the energy estimate for $u_k$ (coming from the assumed $\mathcal{H}^{\infty}_{\theta}(\R)$ well-posedness) we get
	\begin{align}\label{eq_estimate_22}
		\|D^{r}_{x} (1-\chi_k(D)) v^{(\alpha\beta)}_{k}\|\leq C^{\alpha+\beta+M+1} \{\alpha!\beta!\}^{\theta_h} M!^{2\theta_h-1} \nu^{r-pM}_{k}.
	\end{align}
	From \eqref{eq_estimate_11} and \eqref{eq_estimate_22} we conclude the desired estimate.
\end{proof}

\subsubsection{Estimate of $Re\, \langle I_{p-j,k}(x,D) v^{(\alpha\beta)}_{k}, v^{(\alpha\beta)}_{k} \rangle$} 
The symbol $I_{p-j,k}$ defined by \eqref{i2k} belongs to $SG^{p-j,-\sigma_{p-j}}(\R^{2})$ and its seminorms are uniformly bounded with respect to $k$. Indeed, since $x$ is comparable with $\nu^{p-1}_{k}$ and $\xi$ is comparable with $\nu_k$ on the support of $\psi_k(x)\chi_k(\xi)$ we obtain 
\begin{align*}
	|\partial^{\gamma}_{\xi}\partial^{\delta}_{x} &I_{p-j,k}(x,\xi)| \leq \sum_{\overset{\gamma_1+\gamma_2 = \gamma}{\delta_1+\delta_2 = \delta}} \frac{\gamma!\delta!}{\gamma_1!\gamma_2!\delta_1!\delta_2!} 
	|\partial^{\gamma_1}_{\xi} \partial^{\delta_1}_{x}\{ A\langle x \rangle^{-\sigma_{p-j}} \xi^{p-j} - c_{p-j}\langle \nu^{p-1}_{k} \rangle^{-\sigma_{p-j}} \nu^{p-1}_{k}\}|
	\\ &\times|\partial^{\delta_2}_{x}\psi_{k}(x) \partial^{\gamma_2}_{\xi}\chi_{k}(\xi)| \\
	&\leq \sum_{\overset{\gamma_1+\gamma_2 = \gamma}{\delta_1+\delta_2 = \delta}} \frac{\gamma!\delta!}{\gamma_1!\gamma_2!\delta_1!\delta_2!} C^{\gamma_1+\delta_1+1} \gamma_1!\delta_1! \langle \xi \rangle^{p-j - \gamma_1} \langle x \rangle^{-\sigma_{p-j}-\delta_1} C^{\gamma_2+\delta_2+1} \gamma_2!^{\theta_h} \delta_2!^{\theta_h} \nu^{-\gamma_2-\delta_2(p-1)}_{k} \\
	&\leq C^{\gamma+\delta+1} \{\gamma!\delta!\}^{\theta_h} \langle \xi \rangle^{p-j-\gamma} \langle x \rangle^{-\sigma_{p-j}-\delta}.
\end{align*} 
Thanks to \eqref{positive}, we can apply Theorem \ref{theorem_SG_sharp_garding} to $I_{p-j,k}(x,D)$ to get the following decomposition 
$$
	I_{p-j,k}(x,D) = p_{p-j,k}(x,D) + r_{p-j,k}(x,D),
$$
where $p_{p-j,k}(x,D)$ is a positive operator and $r_{p-j,k} \in SG^{p-j-1,-\sigma_{p-j}-1}(\R^{2})$. In addition, since the seminorms of $I_{p-j,k}$ are uniformly bounded with respect to $k$, the same holds for the seminorms of $r_{p-j,k}$. In this way we conclude that
\begin{align*}
	\text{Re}\, \langle I_{p-j,k}(x,D) v^{(\alpha\beta)}_{k}, v^{(\alpha\beta)}_{k} \rangle &\geq \text{Re}\, \langle r_{p-j,k}(x,D) v^{(\alpha\beta)}_{k}, v^{(\alpha\beta)}_{k} \rangle \\
	&=  \text{Re}\, \langle  \underbrace{r_{p-j,k}(x,D) \langle x \rangle^{\sigma_{p-j}+1}}_{\text{order}\, p-j-1} \langle x \rangle^{-\sigma_{p-j}-1} v^{(\alpha\beta)}_{k}, v^{(\alpha\beta)}_{k} \rangle \\
	&\geq - C \|\langle x \rangle^{-\sigma_{p-j}-1} v^{(\alpha\beta)}_{k}\|_{H^{p-j-1}(\R)}  \| v^{(\alpha\beta)}_{k}\| \\
	&\geq - C \| v^{(\alpha\beta)}_{k}\| \sum_{r=0}^{p-j-1} \|D^{r}_{x} \{\langle x \rangle^{-1-\sigma_{p-j}} v^{(\alpha\beta)}_{k}\} \|.
\end{align*}

To handle with the terms $\| D^{r}_{x} \{\langle x \rangle^{-\sigma_{p-j}-1} v^{(\alpha\beta)}_{k}\}\|$ we first write 
\begin{align*}
	D^{r}_{x} \{\langle x \rangle^{-\sigma_{p-j}-1} v^{(\alpha\beta)}_{k}\} &= \sum_{r' = 0}^r \binom{r}{r'} D^{r'}_{x}\langle x \rangle^{-\sigma_{p-j}-1} D^{r-r'}_{x} v^{(\alpha\beta)}_{k}.
\end{align*}
On the support of $D^{r-r'}_{x}v^{(\alpha\beta)}_{k}$ we have that $x$ is comparable with $\nu^{p-1}_{k}$, so
\begin{align*}
	\|D^{r}_{x}\{ \langle x \rangle^{-\sigma_{p-j}-1} v^{(\alpha\beta)}_{k} \} \|\leq C^{r+1} \nu^{-(p-1)(\sigma_{p-j}+1)}_{k} \sum_{r' = 0}^r \frac{r!}{(r-r')!} \|D^{r-r'}_{x} v^{(\alpha\beta)}_{k} \|.
\end{align*}
Applying Lemma \ref{lemma_useful_estimate} for $M = N_k$ we obtain 
\begin{align*}
\|D^{r}_{x}\{ \langle x \rangle^{-\sigma_{p-j}-1} v^{(\alpha\beta)}_{k} \} \| \leq & \, C \nu^{-(p-1)(\sigma_{p-j}+1)+r}_{k} \| v_k^{(\alpha \beta)}\|
 \\ &+ C^{\alpha+\beta+N_k+1} \{\alpha!\beta!\}^{\theta_h} N_k!^{2\theta_h-1} \nu^{-(p-1)(\sigma_{p-j}+1)}_{k} \nu^{r-N_k}_{k}.
\end{align*}
Hence, 
\begin{align}\label{eq_estimate_I_2_k}
	\textrm{Re} \, \langle I_{p-j,k}(x,D) v^{(\alpha\beta)}_{k}, v^{(\alpha\beta)}_{k} \rangle &\geq  -C \nu^{-(p-1)\sigma_{p-j} - j}_{k}\|v^{(\alpha\beta)}_{k}\|^{2}  \\ \nonumber
	&-C^{\alpha+\beta+N_k+1} \{\alpha!\beta!\}^{\theta_h} N_k!^{2\theta_h-1} \nu^{-(p-1)\sigma_{p-j} -j-N_k}_{k} \|v^{(\alpha\beta)}_{k}\|.
\end{align}

\subsubsection{Estimate of $\textrm{Re} \, \langle J_{p-j,k}(t,x,D) v^{(\alpha\beta)}_{k}, v^{(\alpha\beta)}_{k} \rangle$}
Since the supports of $1-\psi_{k}(x)\chi_{k}(\xi)$ and of $w^{(\alpha\beta)}_{k}(x,\xi)$ are disjoint, we may write
\begin{align*}
	J_{p-j,k}(t,x,D) w^{(\alpha\beta)}_{k}(x,D) = R^{(\alpha\beta)}_{p-j,k}(t,x,D), 
\end{align*} 
where 
\begin{align*}
	R^{(\alpha\beta)}_{p-j,k}(t,x,\xi) = N_k\int_{0}^{1} \frac{(1-\theta)^{N_k-1}}{N_k!} \, Os - \iint e^{-iy\eta} \partial^{N_k}_{\xi} J_{p-j,k}(t,x,\xi+\theta\eta) D^{N_k}_{x}w^{(\alpha\beta)}_{k}(x+y,\xi)  dy\dslash\eta \, d\theta.
\end{align*}
The seminorms of $R^{(\alpha\beta)}_{k}$ can be estimated in the following way: for every $\ell_0 \in \N_0$ there exists $\ell_1 = \ell_1(\ell_0)$ such that  
\begin{align*}
	|R^{(\alpha\beta)}_{p-j,k}|^{(0)}_{\ell_0} \leq C(\ell_0) \frac{N_k}{N_k!} |\partial^{N_k}_{\xi} J_{p-j,k}|^{(0)}_{\ell_1} |\partial^{N_k}_{x}w^{(\alpha\beta)}_{k}|^{(0)}_{\ell_1}.
\end{align*}
From Lemma \ref{lemma_estimates_w_k} we get
$$
|\partial^{N_k}_{x}w^{(\alpha\beta)}_{k}|^{(0)}_{\ell_1} \leq C^{\ell_1+\alpha+\beta+N_k+1} \ell_1!^{2\theta_h} \{\alpha!\beta!N_k!\}^{\theta_h} \nu^{-N_k(p-1)}_{k}.
$$
On the other hand, with similar computations as in \eqref{eq_annoying_term} and since there is no harm into assuming $N_k \geq p$ (because $\nu_k \to +\infty$ and $N_k$ is defined by \eqref{eq_definition_of_N_k}), 
\begin{align*}
\partial^{N_k}_{\xi} J_{p-j,k}(t,x,\xi) &= -\sum_{\overset{N_1+N_2=N_k}{N_1 \leq p-j}} \frac{N_k!}{N_1!N_2!} \partial^{N_1}_{\xi}\{ \text{Im}\,a_{p-j}(t,x) \xi^{p-j} - c_{p-j}\langle \nu^{p-1}_{k} \rangle^{-\sigma_{p-j}} \nu^{p-j}_{k}\} \psi_{k}(x)\partial^{N_2}_{\xi}\chi_{k}(\xi)
\end{align*}
hence 
$$
|\partial^{N_k}_{\xi}J_{p-j,k}(t,x,\xi)|^{(0)}_{\ell_1} \leq C^{\ell_1+N_k+1}\ell_{1}!^{2\theta_h}N_k!^{\theta_h} \nu^{p-j-N_k}_{k}.
$$
Therefore 
$$
|R^{(\alpha\beta)}_{p-j,k}|^{(0)}_{\ell_0} \leq C^{\alpha+\beta+N_k+1} \{\alpha!\beta!\}^{\theta_h} N_k!^{2\theta_h-1} \nu^{p-j-N_k}_{k} \nu^{-N_k(p-1)}_{k},
$$
which allows us to conclude
\begin{align*}
	\|J_{p-j,k}(t,x,D) v^{(\alpha\beta)}_{k}\|_ \leq C^{\alpha+\beta+N_k+1} \{\alpha!\beta!\}^{\theta_h} N_k!^{2\theta_h-1} \nu^{p-j-N_k}_{k} \nu^{-N_k(p-1)}_{k}.
\end{align*}
So,
\begin{align}\label{eq_estimate_I_3_k}
	\textrm{Re} \, \langle J_{p-j,k}(t,x,D) v^{(\alpha\beta)}_{k}, v^{(\alpha\beta)}_{k} \rangle \geq 
	-C^{\alpha+\beta+N_k+1} \{\alpha!\beta!\}^{\theta_h} N_k!^{2\theta_h-1} \nu^{p-j-N_k}_{k} \nu^{-N_k(p-1)}_{k} \|v^{(\alpha\beta)}_{k}\|.
\end{align}

\subsubsection{Estimate of \, $\sum_{s=0}^{p-j-1} \binom{p-j}{s} \langle i D^{p-j-s}_{x}\text{Re}\, a_{p-j} D^{s}_{x} v^{(\alpha\beta)}_{k}, v^{(\alpha\beta)}_{k} \rangle$}
From assumption (ii), the support of $D^{s}_{x} v^{(\alpha\beta)}_{k}$ and Lemma \ref{lemma_useful_estimate} we get, for $s \leq p-j-1$,
\begin{align*}
\| D^{p-j-s}_{x}\text{Re}\, a_{p-j} D^{s}_{x} v^{(\alpha\beta)}_{k} \| &\leq C \nu^{-(p-j-s)(p-1)}_{k} \|D^{s}_{x} v^{(\alpha\beta)}_{k} \| \\
&\leq C \nu^{-(p-j-s)(p-1)}_{k}\{ C \nu^{s}_{k} \|v^{(\alpha\beta)}_{k}\| + C^{\alpha+\beta+N_k+1}(\alpha!\beta!)^{\theta_h}N_k!^{2\theta_h - 1}\nu^{s-N_k}_{k} \} \\
&\leq C \nu^{-j}_{k} \|v^{(\alpha\beta)}_{k}\| +  C^{\alpha+\beta+N_k+1}(\alpha!\beta!)^{\theta_h}N_k!^{2\theta_h - 1}\nu^{-j-N_k}_{k}.
\end{align*}
Hence 
\begin{align}\label{eq_estimate_real_part}
	\frac{1}{2}\sum_{s=0}^{p-j-1} \binom{p-j}{s} \langle i D^{p-j-s}_{x}\text{Re}\, a_{p-j} D^{s}_{x} v^{(\alpha\beta)}_{k}, v^{(\alpha\beta)}_{k} \rangle 
	&\geq -C \nu^{-j}_{k} \|v^{(\alpha\beta)}_{k}\|^{2} \\ \nonumber
	&- C^{\alpha+\beta+N_k+1}(\alpha!\beta!)^{\theta_h}N_k!^{2\theta_h - 1}\nu^{-j-N_k}_{k} \|v^{(\alpha\beta)}_{k}\|.
\end{align}

From \eqref{3termini}, \eqref{eq_estimate_I_2_k}, \eqref{eq_estimate_I_3_k} and \eqref{eq_estimate_real_part} using the fact that $\nu_k \to +\infty$ we obtain the following lemma.

\begin{lemma}\label{lemma_estimate_that_gives_contradiction}
	If \eqref{cpp} is well-posed in $\mathcal{H}^{\infty}_{\theta}(\R)$, then for all $k$ sufficiently large the estimate  
	\begin{align*}
		- \sum_{j=1}^{p-1} \text{Re}\, \langle ia_{p-j}(t,x) D^{p-j}_{x} v^{(\alpha\beta)}_{k}, v^{\alpha\beta}_{k} \rangle  &\geq c_1  \nu^{\Xi}_{k} \, \|v^{(\alpha\beta)}_{k}\|^{2} \\
		&-C^{\alpha+\beta+N_k+1} \{\alpha!\beta!\}^{\theta_h} N_k!^{2\theta_h-1} \nu^{-1-N_k}_{k} \|v^{(\alpha\beta)}_{k}\|,	
	\end{align*}
	holds for some $c_1 > 0$ independent of $k, \alpha, \beta$ and $N_k$, where $\Xi$ is defined by \eqref{equation_thesis_main_theorem}.
\end{lemma}

\subsection{Estimate from below of $f^{(\alpha\beta)}_{k}$}\label{subs_f} We start recalling that
$$
f^{(\alpha\beta)}_{k} = [P, w^{(\alpha\beta)}_{k}]u_k = [D_t + a_p(t)D^p_x, w^{(\alpha\beta)}_{k}]u_k  +\sum_{j=1}^{p} \, [a_{p-j}(t,x) D^{p-j}_{x}, w^{(\alpha\beta)}_{k}]u_k.
$$
In the sequel we shall explain how to estimate the above commutators. 

\subsubsection{Estimate of $[D_t + D^p_x, w^{(\alpha\beta)}_{k}]u_k$}\label{subsection_first_braket_f_k_alpha_beta}
Since $w_k^{(\alpha \beta)}$ is independent of $t$ we have
\begin{align*}
	[D_t + a_p(t)D^p_x, w^{(\alpha\beta)}_{k}] &=  a_{p}(t)\sum_{\gamma=1}^{p}\frac{1}{\gamma!} \frac{p!}{(p-\gamma)!} D^{\gamma}_x w^{(\alpha\beta)}_{k} D^{p-\gamma}_{x}.\\
\end{align*}
Now we use the formula 
\begin{equation} \label{inverseLeibniz}
f(x) D^{\lambda}g(x)= \sum_{j=0}^{\lambda} (-1)^{j} \binom{\lambda}{j} D^{\lambda-j} (g(x)D^{j}f(x)) \end{equation}
 for smooth functions $f,g$, and the following elementary identities:
 $$ \sum_{\gamma = 1}^{p} \sum_{j=0}^{p-\gamma} c_{\gamma,j} a_{\gamma+j} = \sum_{\ell=1}^{p} \Bigg\{ \sum_{s=1}^{\ell} c_{s,\ell-s} \Bigg\} a_{\ell},
\qquad \frac{(-1)^{\ell+1}}{\ell!} = \sum_{s=1}^{\ell} \frac{(-1)^{\ell-s}}{s!(\ell-s)!},
$$
to deduce that
\begin{align*}
	[D_t + a_p(t)D^p_x, w^{(\alpha\beta)}_{k}] &=  a_p(t)\sum_{\gamma=1}^{p} \sum_{j=0}^{p-\gamma} \frac{(-1)^{j}p!}{\gamma!j!(p-\gamma-j)!} \left( \frac{1}{i\nu^{p-1}_{k}} \right)^{\gamma+j} D^{p-\gamma-j}_{x} \circ w^{((\alpha+\gamma+j)\beta)}_{k}(x,D) \\
	&= a_{p}(t)\sum_{\ell=1}^{p} \Bigg\{\sum_{s = 1}^{\ell} \frac{(-1)^{\ell-s}}{s!(\ell-s)!}\Bigg\} \frac{p!}{(p-\ell)!} \left(\frac{1}{i\nu_{k}^{p-1}}\right)^{\ell} D^{p-\ell}_{x} \circ w^{((\alpha+\ell)\beta)}_{k}(x,D) \\
	&= a_{p}(t)\sum_{\ell = 1}^{p} (-1)^{\ell+1} \binom{p}{\ell} \left(\frac{1}{i\nu_k^{p-1}}\right)^{\ell} D^{p-\ell}_{x} \circ w^{((\alpha+\ell)\beta)}_{k}(x,D).
\end{align*}
Thus, from Lemma \ref{lemma_useful_estimate} and using the fact that $a_p(t)$ is continuous on $[0,T]$ we get
\begin{align}\label{estimate_first_bracket_f_k_alpha_beta}
	\|[D_t+a_p(t)D^p_x, w^{(\alpha\beta)}_{k}]u_k\| &\leq C \sum_{\ell = 1}^{p} \binom{p}{\ell} \left(\frac{1}{\nu_{k}^{p-1}}\right)^{\ell} 
	\|D^{p-\ell}_{x} v^{((\alpha+\ell)\beta)}_{k}\| \nonumber \\
	&\leq C \sum_{\ell = 1}^{p} \frac{1}{\nu^{p(\ell-1)}_{k}} \|v^{((\alpha+\ell)\beta)}_{k}\|
	+ C^{\alpha+\beta+N_k+1} (\alpha!\beta!)^{\theta_h} N_k!^{2\theta_h-1} \nu^{p-1-N_k}_{k}.
\end{align}

\subsubsection{Estimate of $[ia_{p-j}(t,x)D^{p-j}_{x}, w^{(\alpha\beta)}_{k}]u_k$}
We first observe that
\begin{align}\label{partenza}
	[ia_{p-j}(t,x)D^{p-j}_{x}, w^{(\alpha\beta)}_{k}] &= ia_{p-j}(t,x) \sum_{\gamma = 1}^{p-j} \binom{p-j}{\gamma} D^{\gamma}_{x} w^{(\alpha\beta)}_{k}(x,D) D^{p-j-\gamma}_{x} \\\nonumber
	&-i \sum_{\gamma=1}^{N_k-1} \frac{1}{\gamma!}  D^{\gamma}_{x}a_{p-j}(t,x) \partial^{\gamma}_{\xi}w^{(\alpha\beta)}_{k}(x,D) D^{p-j}_{x}  + r^{(\alpha\beta)}_{p-j,k}(t,x,D),
\end{align}
where 
$$
r^{(\alpha\beta)}_{p-j,k}(t,x,\xi) = -iN_k\int_{0}^{1} \frac{(1-\theta)^{N_k-1}}{N_k!} \, Os - \iint e^{-iy\eta} \partial^{N_k}_{\xi}w^{(\alpha\beta)}_{k}(x,\xi+\theta\eta) D^{N_k}_{x} a_{p-j}(t,x+y) \xi^{p-j} dy\dslash\eta \, d\theta.
$$
To estimate $r^{(\alpha\beta)}_{p-j,k}$ we need to use the support properties of $w^{(\alpha\beta)}_{k}$, then we write 
$$
\xi^{p-j} = (\xi + \theta\eta - \theta\eta)^{p-j} = \sum_{\ell = 0}^{p-j} \binom{p-j}{\ell} (\xi+\theta\eta)^{\ell} (-\theta\eta)^{p-j-\ell}.
$$
Now, integration by parts gives
\begin{align*}
	&Os-\iint e^{-iy\eta} \partial^{N_k}_{\xi}w^{(\alpha\beta)}_{k}(x,\xi+\theta\eta) D^{N_k}_{x} a_{p-j}(t,x+y) \xi^{p-j} dy\dslash\eta \\
	&= \sum_{\ell = 0}^{p-j} \binom{p-j}{\ell} \theta^{p-j-\ell}\, Os-\iint D^{p-j-\ell}_{y}e^{-iy\eta} \partial^{N_k}_{\xi}w^{(\alpha\beta)}_{k}(x,\xi+\theta\eta) (\xi + \theta\eta)^{\ell} D^{N_k}_{x} a_{p-j}(t,x+y) dy\dslash\eta \\
	&= \sum_{\ell = 0}^{p-j} \binom{p-j}{\ell} (-\theta)^{p-j-\ell}\, Os-\iint e^{-iy\eta} \partial^{N_k}_{\xi}w^{(\alpha\beta)}_{k}(x,\xi+\theta\eta) (\xi + \theta\eta)^{\ell} D^{N_k+p-j-\ell}_{x} a_{p-j}(t,x+y) dy\dslash\eta.
\end{align*}
Hence, using assumption (ii), we estimate the seminorms of $r^{(\alpha\beta)}_{k}$ in the following way
\begin{align*}
|r^{(\alpha\beta)}_{p-j,k}|^{0}_{\ell_0} &\leq \frac{C^{N_k+p-j}}{N_k!} \sum_{\ell = 0}^{p-j} |\xi^{\ell}\partial^{N_k}_{\xi}w^{(\alpha\beta)}_{k}|^{(0)}_{\ell_1} |D^{N_k+p-j-\ell}_{x}a_{p-j}(t,x)|^{(0)}_{\ell_1} \\
&\leq C^{\alpha+\beta+N_k+1} \{\alpha!\beta!\}^{\theta_h} N_k!^{2\theta_h-1} \nu^{p-j-N_k}_{k},
\end{align*}
which allows to conclude (using the assumed $\mathcal{H}^{\infty}_{\theta}$ well-posedness) that
 \begin{equation}\label{eq_estimate_main_remainder_f_k}
 \|r^{(\alpha\beta)}_{p-j,k}(\cdot, D)u_{k}\| \leq C^{\alpha+\beta+N_k+1}\{\alpha!\beta!\}^{\theta_h} N_k!^{2\theta_h-1}\nu^{p-j-N_k}_{k}.
 \end{equation}

Now we consider the second term in \eqref{partenza}, and by \eqref{inverseLeibniz} we write 
\begin{align*}
	\sum_{\gamma=1}^{N_k-1} &\frac{1}{\gamma!}  D^{\gamma}_{x}a_{p-j}(t,x) \partial^{\gamma}_{\xi}w^{(\alpha\beta)}_{k}(x,D) D^{p-j}_{x} \\
	&= \sum_{\gamma = 1}^{N_k-1} \sum_{s=0}^{p-j} \frac{1}{\gamma!} \binom{p-j}{s} \left(\frac{4}{\nu_k}\right)^{\gamma} \left(\frac{i}{\nu^{p-1}_{k}}\right)^{s} 
	D^{\gamma}_{x} a_{p-j}(t,x) D^{p-j-s}_{x} \circ w^{((\alpha+s)(\beta+\gamma))}_k(x,D).
\end{align*}
Using the support of $D^{p-j-s}_{x} v^{((\alpha+s)(\beta+\gamma))}_{k}$ and the assumption (ii) we get
\begin{align*}
	\| D^{\gamma}_{x} a_{p-j}(t,x) D^{p-j-s}_{x} v^{((\alpha+s)(\beta+\gamma))}_{k} \| \leq 
	C^{\gamma+1} \gamma! \nu^{-\gamma(p-1)}_{k} \|D^{p-j-s}_{x}v^{((\alpha+s)(\beta+\gamma))}_{k}\|.
\end{align*}
Now, applying Lemma \ref{lemma_useful_estimate} with $M = N_k - \gamma$  we obtain 
\begin{align*}
	\| D^{\gamma}_{x} a_{p-j}(t,x) D^{p-j-s}_{x} &v^{((\alpha+s)(\beta+\gamma))}_{k} \| \leq C^{\gamma+1} \gamma! \nu^{-\gamma(p-1)}_{k} \nu^{p-j-s}_{k} 
	\|v^{((\alpha+s)(\beta+\gamma))}_{k}\| \\
	&+ C^{\gamma+1} \gamma! \nu^{-\gamma(p-1)}_{k} C^{\alpha+s+\beta+N_k} \{\alpha!\beta!s!\gamma!\}^{\theta_h} (N_k-\gamma)!^{2\theta_h-1} \nu^{p-j-s-(N_k-\gamma)}_{k}.
\end{align*}
Hence
\begin{align}\label{eq_estimate_shift_gamma_f_k}
	\Bigg\|\sum_{\gamma=1}^{N_k-1} &\frac{1}{\gamma!}  D^{\gamma}_{x} a_{p-j}(t,x) \partial^{\gamma}_{\xi}w^{(\alpha\beta)}_{k}(x,D) D^{p-j}_{x} u_k\Bigg\| \leq \\ &\leq C\sum_{s=0}^{p-j}\sum_{\gamma=1}^{N_k-1} C^{\gamma} \nu^{-(p-j)(s+\gamma-1)-j(s+\gamma)}_{k} \|v^{((\alpha+s)(\beta+\gamma))}_{k}\|\nonumber + C^{\alpha+\beta+N_k+1} \{\alpha!\beta!\}^{\theta_h} N_k!^{2\theta_h-1}\nu_k^{1-j-N_k}.
\end{align}

Finally, for the first term in \eqref{partenza}, we follow the same steps of Subsection \ref{subsection_first_braket_f_k_alpha_beta} to get 
\begin{align*}
	\sum_{\gamma = 1}^{p-j} \binom{p-j}{\gamma} D^{\gamma}_{x} w^{(\alpha\beta)}_{k}(x,D) D^{p-j-\gamma}_{x} =  \sum_{\ell = 1}^{p-j} (-1)^{\ell+1} \binom{p-j}{\ell} \frac{1}{(i\nu^{p-1}_{k})^{\ell}}D^{p-j-\ell}_{x} \circ w^{((\alpha+\ell)\beta)}_{k}(x,D),
\end{align*}
and therefore 
\begin{align}\label{eq_last_estimate_f_k_alpha_beta}
	\Bigg\|ia_{p-j}(t,x) \sum_{\gamma = 1}^{p-j} &\binom{p-j}{\gamma} D^{\gamma}_{x} w^{(\alpha\beta)}_{k}(x,D) D^{p-j-\gamma}_{x} u_k \Bigg\| \leq \\ \nonumber
	 &\leq C \sum_{\ell=1}^{p-j} \frac{1}{\nu^{j+p(\ell-1)}_{k}} \|v^{((\alpha+\ell)\beta)}_{k}\|
	 +C^{\alpha+\beta+N_k+1}\{\alpha!\beta!\}^{\theta_h} N_k!^{2\theta_h-1} \nu^{1-j-N_k}_{k}.
\end{align}

Gathering \eqref{eq_estimate_main_remainder_f_k}, \eqref{eq_estimate_shift_gamma_f_k} and \eqref{eq_last_estimate_f_k_alpha_beta} we obtain 
\begin{align}\label{electric_eye}
	\sum_{j=1}^{p} \|[ia_{p-j}(t,x)&D^{p-j}_{x},w^{(\alpha\beta)}_{j}]u_k\| \leq C \sum_{j=1}^{p} \sum_{s = 0}^{p-j} \sum_{\gamma = 1}^{N_k - 1} C^{\gamma} \nu^{-(p-j)(s+\gamma-1)-j(s+\gamma)}_{k}\|v^{((\alpha+s)(\beta+\gamma))}_{k}\| \\\nonumber
	&+ C\sum_{j=1}^{p}\sum_{\ell=1}^{p-j} \frac{1}{\nu^{j+p(\ell-1)}_{k}} \|v^{((\alpha+\ell)\beta)}_{k}\| + C^{\alpha+\beta+N_k+1} \{\alpha!\beta!\}^{\theta_h} N_k!^{2\theta_h-1} \nu^{p-1-N_k}_{k} \\\nonumber
	&\leq C \sum_{s = 0}^{p-1} \sum_{\gamma = 1}^{N_k - 1} C^{\gamma} \nu^{-(p-1)(s+\gamma-1)-(s+\gamma)}_{k} \|v^{((\alpha+s)(\beta+\gamma))}_{k} \|  \\\nonumber
	&+ C\sum_{\ell=1}^{p-1} \frac{1}{\nu^{1+p(\ell-1)}_{k}} \|v^{((\alpha+\ell)\beta)}_{k}\| + C^{\alpha+\beta+N_k+1} \{\alpha!\beta!\}^{\theta_h} N_k!^{2\theta_h-1} \nu^{p-1-N_k}_{k}.
\end{align}
Now, combining \eqref{estimate_first_bracket_f_k_alpha_beta} and \eqref{electric_eye} we can summarize the computations of Subsection \ref{subs_f} in the following lemma.

\begin{lemma}\label{lemma_estimate_f_k_alpha_beta}
	If \eqref{cpp} is well-posed in $\mathcal{H}^{\infty}_{\theta}(\R)$, then 
		\begin{align*}
			\|f^{(\alpha\beta)}_{k}\|&\leq C \sum_{\ell = 1}^{p} \frac{1}{\nu^{p(\ell-1)}_{k}} \|v^{((\alpha+\ell)\beta)}_{k}\|+ C\sum_{s=0}^{p-1}\sum_{\gamma=1}^{N_k-1} C^{\gamma} \nu^{-(p-1)(s+\gamma-1)-s-\gamma}_{k} \|v^{((\alpha+s)(\beta+\gamma))}_{k}\| \\ \nonumber
			&+ C^{\alpha+\beta+N_k+1} (\alpha!\beta!)^{\theta_h} N_k!^{2\theta_h-1} \nu^{p-1-N_k}_{k}
		\end{align*}	 
	for some constant $C > 0$ independent from $k, \alpha, \beta$ and $N_k$.
\end{lemma}

\subsection{A lower bound estimate for $\partial_t E_{k}(t)$.} Lemmas \ref{lemma_estimate_that_gives_contradiction}, \ref{lemma_estimate_f_k_alpha_beta} and \eqref{e} give the following
\begin{align*}
	\partial_t \|v^{(\alpha\beta)}_{k}\|&\geq c_1 \nu^{\Xi}_{k} \|v^{(\alpha\beta)}_{k}\| - C_{a_0} \|v^{(\alpha\beta)}_{k}\|
	- C \sum_{\ell = 1}^{p} \frac{1}{\nu^{p(\ell-1)}_{k}} \|v^{((\alpha+\ell)\beta)}_{k}\| \\
	&- C\sum_{s=0}^{p-1}\sum_{\gamma=1}^{N_k-1} C^{\gamma} \nu^{-(p-1)(s+\gamma-1)-s-\gamma}_{k} \|v^{((\alpha+s)(\beta+\gamma))}_{k}\| - C^{\alpha+\beta+N_k+1} (\alpha!\beta!)^{\theta_h} N_k!^{2\theta_h-1} \nu^{p-1-N_k}_{k}.
\end{align*}
Therefore
\begin{align}\label{ripartenza}
	\partial_t E_{k}(t) &= \sum_{\alpha \leq N_k, \beta \leq N_k} \frac{1}{(\alpha!\beta!)^{\theta_1}} \partial_t \| v^{(\alpha\beta)}_{k} (t,\cdot) \| \\\nonumber
	&\geq \sum_{\alpha \leq N_k, \beta \leq N_k} \frac{1}{(\alpha!\beta!)^{\theta_1}} \{c_1 \nu^{\Xi}_{k} - C_{a_0} \}\|v^{(\alpha\beta)}_{k}\|-C \sum_{\ell = 1}^{p} \frac{1}{\nu^{p(\ell-1)}_{k}} \sum_{\alpha \leq N_k, \beta \leq N_k} \frac{1}{(\alpha!\beta!)^{\theta_1}}   \|v^{((\alpha+\ell)\beta)}_{k}\| \\\nonumber
	&-C \sum_{s=0}^{p-1}\sum_{\gamma=1}^{N_k-1} C^{\gamma} \nu^{-(p-1)(s+\gamma-1)-s-\gamma}_{k} \sum_{\alpha \leq N_k, \beta \leq N_k} \frac{1}{(\alpha!\beta!)^{\theta_1}} \|v^{((\alpha+s)(\beta+\gamma))}_{k}\|\\\nonumber
	&-\sum_{\alpha \leq N_k, \beta \leq N_k} \frac{1}{(\alpha!\beta!)^{\theta_1}} C^{\alpha+\beta+N_k+1} (\alpha!\beta!)^{\theta_h} N_k!^{2\theta_h-1} \nu^{p-1-N_k}_{k}.
\end{align}

In the sequel we shall discuss how to treat all the terms appearing in the above summation. For the first one, by \eqref{eq_def_energy} and \eqref{vk} we simply have
\begin{align}\label{eq_estimate_first_term}
	\sum_{\alpha \leq N_k, \beta \leq N_k} \frac{1}{(\alpha!\beta!)^{\theta_1}} \{c_1 \nu^{\Xi}_{k} - C_{a_0} \}\|v^{(\alpha\beta)}_{k}\|= \{c_1 \nu^{\Xi}_{k} - C_{a_0} \}E_k(t).
\end{align}
For the second one we proceed as follows
\begin{align*}
	\sum_{\ell = 1}^{p} \frac{C}{\nu^{p(\ell-1)}_{k}} \sum_{\alpha \leq N_k, \beta \leq N_k} \frac{1}{(\alpha!\beta!)^{\theta_1}}   \|v^{((\alpha+\ell)\beta)}_{k}\| &= 
	\sum_{\ell = 1}^{p} \frac{C}{\nu^{p(\ell-1)}_{k}} \sum_{\alpha \leq N_k, \beta \leq N_k} \frac{(\alpha+\ell)!^{\theta_1}}{\alpha!^{\theta_1}} E_{k,\alpha+\ell,\beta} \\
	&\leq \sum_{\ell = 1}^{p}\frac{CN_k^{\ell\theta_1}}{\nu^{p(\ell-1)}_{k}} \Bigg\{ E_{k} + \sum_{\alpha = N_k-\ell+1}^{N_k} \sum_{\beta \leq N_k} E_{k,\alpha+\ell,\beta} \Bigg\}.
\end{align*}
Recalling \eqref{Ekab} we get the upper bound
$$
E_{k,\alpha+\ell,\beta} \leq C^{\alpha+\beta+\ell+1} \{(\alpha+\ell)!\beta!\}^{\theta_h-\theta_1},
$$
so, since $\alpha+\ell \geq N_k$, we obtain
\begin{align*}
	\sum_{\ell = 1}^{p} \frac{C}{\nu^{p(\ell-1)}_{k}} \sum_{\alpha \leq N_k, \beta \leq N_k} \frac{1}{(\alpha!\beta!)^{\theta_1}}   \|v^{((\alpha+\ell)\beta)}_{k}\| &\leq 
	\sum_{\ell = 1}^{p}\frac{CN_k^{\ell\theta_1}}{\nu^{p(\ell-1)}_{k}} \Bigg\{ E_{k} + C^{N_k+1} N_k!^{\theta_h-\theta_1} \Bigg\}.
\end{align*}
Now we use the definition of $N_k := \lfloor \nu_{k}^{\frac{\lambda}{\theta_1}} \rfloor$ and the inequality $N_k^{N_k} \leq e^{N_k} N_k!$ to conclude that
\begin{align}\label{eq_estimate_second_term}
	\sum_{\ell = 1}^{p} \frac{C}{\nu^{p(\ell-1)}_{k}} \sum_{\alpha \leq N_k, \beta \leq N_k} \frac{1}{(\alpha!\beta!)^{\theta_1}}   \|v^{((\alpha+\ell)\beta)}_{k}\| &\leq 
	C \sum_{\ell = 1}^{p} \frac{\nu^{\ell\lambda}_{k}}{\nu^{p(\ell-1)}_{k}} \Bigg\{ E_{k} + C^{N_k+1} e^{N_k(\theta_h-\theta_1)} \nu^{\frac{\lambda(\theta_h-\theta_1)N_k}{\theta_1}}_{k} \Bigg\}
\\ &\leq \nonumber
	C \sum_{\ell = 1}^{p} \frac{\nu^{\ell\lambda}_{k}}{\nu^{p(\ell-1)}_{k}} E_{k} + C^{N_k+1} \nu^{C-cN_k}_{k},
\end{align}
where, from now on, $c > 0$ shall denote a constant independent from $k$ and $N$.

For the third term in \eqref{ripartenza} we recall the following inequality
$$
\frac{(\beta+\gamma)!}{\beta!} \leq (\beta+\gamma)^{\gamma} \leq (rN_k)^{\gamma} \leq r^\gamma (\nu^{\frac{\lambda}{\theta_1}}_{k})^{\gamma}, \quad \text{provided that}\, \beta+\gamma \leq rN_k, r \in\N,
$$
thus, since $\lambda \in (0,1)$, for $k$ sufficiently large so that $C\nu^{\lambda-1}_{k} < 1$ we have
\begin{align*}
 \sum_{s=0}^{p-1} &\sum_{\gamma=1}^{N_k-1} C^{\gamma} \nu^{-(p-1)(s+\gamma-1)-s-\gamma}_{k}  \sum_{{\alpha \leq N_k, \beta \leq N_k}} \frac{1}{(\alpha!\beta!)^{\theta_1}} \|v^{((\alpha+s)(\beta+\gamma))}_{k}\| \\
 &=  \sum_{s=0}^{p-1} \sum_{\gamma=1}^{N_k-1}  \Bigg\{ \sum_{\overset{\alpha \leq N_k-s}{\beta \leq N_k-\gamma}} + \sum_{\overset{\alpha \leq N_k, \beta \leq N_k}{\alpha+s > N_k\,\text{or}\,\beta+\gamma > N_k} } \Bigg\}  C^{\gamma}\nu^{-(p-1)(s+\gamma-1)-s-\gamma}_{k} \left(\frac{(\alpha+s)! (\beta+\gamma)!}{\alpha!\beta!} \right)^{\theta_1} E_{k,\alpha+s,\beta+\gamma} \\
 &\leq \sum_{s=0}^{p-1} \sum_{\gamma=1}^{N_k-1} \sum_{{\alpha \leq N_k, \beta \leq N_k}} \Bigg\{ \sum_{\overset{\alpha \leq N_k-s}{\beta \leq N_k-\gamma}} + \sum_{\overset{\alpha \leq N_k, \beta \leq N_k}{\alpha+s > N_k \,\text{or}\, \beta+\gamma > N_k}} \Bigg\} (C\nu^{\lambda-1}_{k})^{\gamma+s} \nu^{-(p-1)(s+\gamma-1)}_{k} E_{k,\alpha+s,\beta+\gamma} \\
 &\leq \Bigg[E_{k} +  \sum_{s=0}^{p-1} \sum_{\gamma=1}^{N_k-1} \sum_{\overset{\alpha \leq N_k, \beta \leq N_k}{\alpha+s > N_k \,\text{or}\, \beta+\gamma > N_k}} C^{\alpha+s+\beta+\gamma+1} \{(\alpha+s)!(\beta+\gamma)!\}^{\theta_h-\theta_1} \Bigg] \\
 &\leq E_{k} + C^{N_k+1} N_k!^{\theta_h-\theta_1}. 
\end{align*}
From the definition of $N_k := \lfloor \nu_{k}^{\frac{\lambda}{\theta_1}} \rfloor$ we obtain 
\begin{align}\label{eq_estimate_third_term}
	\sum_{s=0}^{p-1} \sum_{\gamma=1}^{N_k-1} &C^{\gamma} \nu^{-(p-1)(s+\gamma-1)-s-\gamma}_{k} \sum_{\alpha \leq N_k, \beta \leq N_k} \frac{1}{(\alpha!\beta!)^{\theta_1}} \|v^{((\alpha+s)(\beta+\gamma))}_{k}\|
	\leq  E_{k} + C^{N_k+1} \nu^{C-cN_k}_{k}
\end{align}
for all $k$ large enough, where $C,c$ are positive constants independent of $k$.

Finally, for the last term in \eqref{ripartenza}, using the definition of $N_k := \lfloor \nu_{k}^{\frac{\lambda}{\theta_1}} \rfloor$ and recalling that $\theta_1 > \theta_h$ we easily conclude 
\begin{align}\label{eq_estimate_fourth_term}
	\sum_{\alpha \leq N_k, \beta \leq N_k} \frac{1}{(\alpha!\beta!)^{\theta_1}} C^{\alpha+\beta+N_k+1} (\alpha!\beta!)^{\theta_h} N_k!^{2\theta_h-1} \nu^{p-1-N_k}_{k} 
	\leq C^{N_k+1} \nu^{C-cN_k}_{k}.
\end{align}

From \eqref{eq_estimate_first_term},\eqref{eq_estimate_second_term},\eqref{eq_estimate_third_term} and \eqref{eq_estimate_fourth_term} we obtain the following proposition.

\begin{proposition}\label{prop_from_hell}
	If \eqref{cpp} is well-posed in $\mathcal{H}^{\infty}_{\theta}(\R)$, then for all $t \in [0,T]$ and $k$ sufficiently large the inequality
	\begin{align*}
		\partial_t E_{k}(t) \geq \Bigg[ c_1 \nu^{\Xi}_{k} - C \sum_{\ell = 1}^{p} \frac{\nu^{\ell\lambda}_{k}}{\nu^{p(\ell-1)}_{k}}
		\Bigg] E_{k}(t) - C^{N_k+1}\nu^{C-cN_k}_{k}
	\end{align*}
	holds for some $C, c >0$ independent from $k$ and $N_k$.
\end{proposition}

\section{The proof of Theorem \ref{main_theorem}}\label{section_proof_of_main_thm}
We are now ready to prove our main result. 
\\

\noindent \textit{Proof of Theorem \ref{main_theorem}.} Denote by
\begin{equation}\label{ak}
	A_k :=  c_1 \nu^{\Xi}_{k} - C \sum_{\ell = 1}^{p} \frac{\nu^{\ell\lambda}_{k}}{\nu^{p(\ell-1)}_{k}}, \quad R_k = C^{N_k+1}\nu^{C-cN_k}_{k}.
\end{equation}
By Proposition \ref{prop_from_hell} we have 
$\partial_t E_{k}(t) \geq A_k E_k(t) - R_k$.

Choosing the parameter $\lambda < \min\{\Xi, 1\}$, then we have $(\ell-1)(\lambda-p) + \lambda \leq \lambda < \Xi$, and so $\ell\lambda -p(\ell-1) < \Xi$, and this means that the leading term in $A_k$ is the first one; 
so for $k$ sufficiently large we obtain 
$$
A_{k} \geq \frac{c_1}{2} \nu^{\Xi}_{k}.
$$

In the next estimates we shall always consider $k$ sufficiently large. Applying Gronwall's inequality we obtain
\begin{align}\label{eq_almost_there}
	E_k(t) \geq e^{A_k t} \left(E_k(0) - R_k \int^{t}_{0} e^{- A_k \tau} d\tau  \right) \geq e^{t\frac{c_1}{2} \nu^{\Xi}_{k}} \left( E_k(0) - t R_k  \right),\quad t\in(0,T].
\end{align}

Now we need to estimate properly the term $R_k$ from above and $E_k(0)$ from below. Recalling that $R_k = C \nu_{k}^{C-cN_k}$ and $N_k := \lfloor \nu_{k}^{\frac{\lambda}{\theta_1}} \rfloor$ we may estimate $R_k$ from above in the following way
\begin{equation}\label{eq_almost_there_2}
R_k \leq C e^{-c \nu^{\frac{\lambda}{\theta_1}}_{k}}.
\end{equation}
The estimate for $E_k(0)$ is more involved. We have 
\begin{align*}
E_k(0) \geq \|w_k(x,D)\phi_k\|_{L^2(\R_x)} &= \left\| h\left( \frac{x-4\nu^{p-1}_{k}}{\nu^{p-1}_{k}} \right) h\left( \frac{D - \nu_k}{\frac{1}{4}\nu_{k}} \right) \phi_k \right\|_{L^2(\R_x)} \\
&= \left\| \mathcal{F} \left[ h \left(\frac{x-4\nu^{p-1}_{k}}{\nu^{p-1}_{k}} \right) \right] (\xi) \ast h\left( \frac{\xi-\nu_k}{\frac{1}{4}\nu_{k}} \right) \widehat{\phi}_k(\xi)  \right\|_{L^2(\R_{\xi})} \\
&= \nu^{p-1}_{k} \left\| e^{-4i\nu^{p-1}_{k} \xi}  \,\widehat{h}(\nu^{p-1}_{k} \xi) \ast h\left( \frac{\xi-\nu_k}{\frac{1}{4}\nu_{k}} \right) e^{-4i\nu^{p-1}_{k}\xi} \widehat{\phi}(\xi)  \right\|_{L^2(\R_{\xi})}
\end{align*}
hence 
\begin{equation*}
	E^{2}_k(0) \geq \nu^{2(p-1)}_{k} \int_{\R_{\xi}} \left| \int_{\R_\eta} \widehat{h}(\nu^{p-1}_{k}(\xi - \eta)) h\left( \frac{\eta-\nu_k}{\frac{1}{4}\nu_{k}} \right) \widehat{\phi}(\eta) d\eta
	\right|^{2} d\xi.
\end{equation*}
We choose the Gevrey cutoff function $h$ in such a way that $\widehat{h}(0) > 0$ and $\widehat{h}(\xi) \geq 0$ for all $\xi \in \R$. Then we get an estimate from below to $E_k^2(0)$ by restricting the integration domain. Indeed, we set 
$$
G_{1,k} = \left[\nu_k - \frac{\nu_k}{8}, \nu_k - \frac{\nu_k}{8} + \nu^{-p}_{k}\right] \bigcup \left[\nu_k + \frac{\nu_k}{8} - \nu^{-p}_{k}, \nu_k + \frac{\nu_k}{8}\right],
$$
$$
G_{2,k} = \left[\nu_k - \frac{\nu_k}{8}-\nu^{-p}_{k}, \nu_k - \frac{\nu_k}{8} + \nu^{-p}_{k}\right] \bigcup \left[\nu_k + \frac{\nu_k}{8} - \nu^{-p}_{k}, \nu_k + \frac{\nu_k}{8}+\nu^{-p}_{k}\right].
$$
Then $|\eta - \nu_k| \leq 8^{-1}\nu_k$ for all $\eta \in G_{1,k}$ and $\nu^{p-1}_{k}|\xi-\eta| \leq 2\nu^{-1}_{k}$ for all $\eta \in G_{1,k}$ and all $\xi \in G_{2,k}$. 
Now, if $(\xi,\eta)\in G_{2,k}\times G_{1,k}$, then $\nu^{p-1}_{k}(\xi - \eta)$ is close to zero at least for $k$ large enough, and so by the choice $\widehat{h}(0) > 0$ and $\widehat{h}(\xi) \geq 0$ there exists a positive constant $C$ such that  for $k$ large enough we have $\widehat{h}(\nu^{p-1}_{k}(\xi - \eta))>C$. Moreover, if $\eta\in G_{1,k}$ then $h\left( \frac{\eta-\nu_k}{\frac{1}{4}\nu_{k}} \right) =1$, and finally using the fact that $\hat\phi(\eta)=e^{-2\rho_0\langle\eta\rangle^{\frac1\theta}}$, but $\eta$ is comparable with $\nu_k$ on $G_{1,k}$, we can write $\hat\phi(\eta)\geq  e^{-c_{\rho_0} \nu^{\frac{1}{\theta}}_{k}}$ for a positive constant $c_{\rho_0}$ depending on $\rho_0;$. Summing up, by restricting the domain of integration we have
\begin{align*}
	E^{2}_k(0) &\geq C \nu^{2(p-1)}_{k} e^{-2c_{\rho_0} \nu^{\frac{1}{\theta}}_{k}}\int_{G_{2,k}} \left| \int_{G_{1,k}}  d\eta \right|^{2} d\xi \\
	&= C \nu^{-(p+2)}_{k} e^{-2c_{\rho_{0}} \nu_{k}^{\frac{1}{\theta}}}.
\end{align*}
%
%
In this way we get the following estimate
\begin{align}\label{eq_almost_there_3}
	E_{k}(0) \geq C \nu^{-(p+2)/2}_{k} e^{-c_{\rho_{0}} \nu_{k}^{\frac{1}{\theta}}}.
\end{align}

From \eqref{eq_almost_there}, \eqref{eq_almost_there_2} and \eqref{eq_almost_there_3} we conclude that for all $t \in (0,T]$
\begin{align}\label{estimate_from_below_E_k}
	E_k(t) \geq C e^{t \frac{c_1}{2} \nu_{k}^{\Xi}} \left[ \nu^{-(p+2)/2}_{k} e^{-c_{\rho_0} \nu^{\frac{1}{\theta}}_{k} } - te^{-c \nu^{\frac{\lambda}{\theta_1}}_{k}} \right],
\end{align}
provided that $\lambda < \min\{\Xi, 1\}$ and $k$ being sufficiently large.\\
	Suppose now that \eqref{cpp} is $\mathcal{H}^{\infty}_{\theta}(\R)$ well-posed and  
	assume by contradiction that $\frac{1}{\theta} < \Xi$. Then we take $\lambda = \frac{1}{\theta}+\tilde{\varepsilon}$ with $\tilde{\varepsilon} > 0$ small so that $\lambda < \Xi$  and $\lambda < 1$. After that, we set $\theta_1$ very close to $1$ to get $\frac{\lambda}{\theta_1} > \frac{1}{\theta}$. In this way \eqref{estimate_from_below_E_k} implies for $k$ sufficiently large
	\begin{align*}
		E_k(t) \geq C_2 e^{\tilde{c}_0 \nu_{k}^{\Xi}} e^{-\tilde{c}_{\rho_0} \nu^{\frac{1}{\theta}}_{k} }.
	\end{align*}
	Thus, for $k$ sufficiently large
	\begin{align*}
		E_k(t) \geq C_3 e^{\frac{\tilde{c}_0}{2} \nu_{k}^{\Xi}} \to \infty \quad \textit{for} \quad k \to \infty.
	\end{align*}
	The latter inequality provides a contradiction, since the assumed $\mathcal{H}^{\infty}_{\theta}$ well-posedness implies that the energy $E_k(t)$ is uniformly bounded from above, see \eqref{eq_estimate_from_above_of_energy_k}.
	\qed



\begin{thebibliography}{AAA}

\bibitem{AACgev}
A. Arias Junior, A.Ascanelli, M. Cappiello, {\it Gevrey well posedness for $3$-evolution equations with variable coefficients}, 2022. To appear in Ann. Scuola Norm. Sup. Pisa Cl. Sci. 
\textrm{DOI: 10.2422/2036-2145.202202\_011}, https://arxiv.org/abs/2106.09511

\bibitem{AAC3evolquasilin}
A. Arias Junior, A.Ascanelli, M. Cappiello, {\it KdV-type equations in projective Gevrey spaces}, J. Math. Pures Appl. (2023) In press. https://doi.org/10.1016/j.matpur.2023.07.007

\bibitem{AC_SG_Sharp_Garding}
A. Arias Junior, M. Cappiello, {\it On the Sharp G{\aa}rding Inequality for Operators with Polynomially Bounded and Gevrey Regular Symbols}, Mathematics, \textbf{8} (2020).

 
\bibitem{ascanelli_chiara_zanghirati_2012}
 A. Ascanelli, C. Boiti, L. Zanghirati, {\it Well-posedness of the Cauchy problem for p-evolution equations.} J. Differential Equations \textbf{253} (10) (2012), 2765-2795.

 \bibitem{ascanelli_chiara_zanghirati_necessary_condition_for_H_infty_well_posedness_of_p_evo_equations}
A. Ascanelli, C. Boiti, L. Zanghirati, {\it A Necessary condition for $ H^{\infty}$ well-posedness of $ p $-evolution equations.} Advances in Differential Equations \textbf{21} (2016), 1165-1196.
 

\bibitem{ACJee}
A. Ascanelli, M. Cappiello, \textit{Weighted energy estimates for p-evolution equations in SG classes}, J. Evol.
Eqs, \textbf{15} (2015) n. 3, 583-607.
 
\bibitem{ACR}
A. Ascanelli, M. Cicognani, M.Reissig, {\it The interplay between decay of the data and regularity of the solution in Schr{\"o}dinger equations.} Ann. Mat. Pura Appl. \textbf{199} (2020), n. 4, 1649-1671.
 
\bibitem{CicRei} 
M.~Cicognani, M.~Reissig, \textit{Well-posedness for degenerate Schr\"odinger equations}, Evol. Equ. Control Theory {\bf 3} (2014), no. 1, 15-33.


\bibitem{CRnec}
M. Cicognani, M. Reissig, \textit{Necessity of Gevrey-type Levi conditions for degenerate Schrödinger equations}, J. Abstr. Differ. Equ. Appl. \textbf{5} (2014) n. 1, 52–70.

\bibitem{dreher}
M. Dreher, {\it Necessary conditions for the well-posedness of Schr{\"o}dinger type equations in Gevrey spaces.}
Bull. Sci. Math. \textbf{127} (6) 2003, 485-503.


\bibitem{ichinose_remarks_cauchy_problem_schrodinger_necessary_condition}
W. Ichinose, {\it Some remarks on the Cauchy problem for Schr{\"o}dinger type equations.}
 Osaka J.  Math. \textbf{21} (3) (1984) 565-581.

\bibitem{I2} W.Ichinose, {\it Sufficient condition on $H^{\infty }$ 
well-posedness for Schr\"odinger type equations.}
Comm. Partial Differential Equations, {\textbf 9}, n.1 (1984), 33-48. 

\bibitem{KB}
K. Kajitani, A. Baba, {\it The Cauchy problem for Schr{\"o}dinger type equations.}
Bull. Sci. Math.\textbf{119} (5), (1995), 459-473. 

\bibitem{Kumano-Go}
H. Kumano-Go. \textit{Pseudo-differential operators}. The MIT Press, Cambridge, London, 1982.

\bibitem{mizohata2014}
S. Mizohata, {\it On the Cauchy problem}, Academic Press. Volume 3 (2014).

\bibitem{PZ}
A.D. Polyanin, V.F. Zaitsev, \textit{Handbook of nonlinear partial differential equations}, CRC Press, Boca Raton, FL, 2012.
 
\end{thebibliography}
\end{document}